\documentclass[11pt]{article}
\usepackage[margin=1in]{geometry}
\usepackage[utf8]{inputenc}
\usepackage{amsmath,amsfonts,amssymb,mathrsfs,amstext,amscd,latexsym,amsthm, mathtools}
\usepackage{accents}
\usepackage{hyperref}
\usepackage{graphicx}
\usepackage{subcaption}
\usepackage{xcolor}
\usepackage{comment}
\usepackage{enumitem}
\usepackage[style=numeric]{biblatex}
\newcommand{\pI}[1]{\left <{#1}\right >}
\newcommand{\set}[1]{\left\{#1\right\}}
\def\R{\mathbb{R}}
\def\e{\mathbf{e}}

\newcommand\restr[2]{{% we make the whole thing an ordinary symbol
		\left.\kern-\nulldelimiterspace % automatically resize the bar with \right
		#1 % the function
		\vphantom{\big|} % pretend it's a little taller at normal size
		\right|_{#2} % this is the delimiter
}}
\newtheorem{theorem}{Theorem}[section]
\newtheorem{lemma}[theorem]{Lemma}

\newtheorem{proposition}[theorem]{Proposition}
\newtheorem{definition}{Definition}
\newtheorem{remark}{Remark}
\newtheorem{step}{Step}[section]
\newtheorem{example}{Example}

\title{Bowl Soliton Asymptotics and Applications}
\author{Sathya Rengaswami, Jos\'e Torres Santaella}
\date{}

\begin{document}

\maketitle
\begin{abstract}
    In this paper, we obtain the asymptotic expansion for the analogue of the bowl-soliton for a large `nondegenerate' class of fully nonlinear curvature flows. We use this to show the uniqueness of these bowl-type solitons in their asymptotic class. We also give examples to illustrate the situation for `degenerate' speeds and how they different they can be. Finally, we show how to construct `wing-like' solitons for these flows, which are complete, connected translators that are not graphical, entire or convex. We also obtain asymptotic expansions for them to show the variety of solutions that one can obtain depending on the choice of speed function.
\end{abstract}
\section{Introduction}\label{Intro}\,
Geometric evolution equations for hypersurfaces have seen significant development over the last few decades. We have witnessed a remarkable growth in this field, leading to the emergence of intriguing nonlinear partial differential equations. These equations have played a crucial role in addressing fundamental questions within both mathematics and physics.
\newline

In this research, our focus lies on a particular type of evolving hypersurfaces known as "translators." These hypersurfaces in $\mathbb{R}^{n+1}$ undergo evolution by translation along a fixed unit direction when subjected to an extrinsic curvature flow, i.e. the normal velocity at each point of the hypersurface is a 1-homogeneous smooth symmetric function of their principal curvatures. It's worth noting that translators represent a significant class of second-order elliptic partial differential equations.
\newline

More precisely, given an immersed hypersurface $\Sigma_0=F_0(\Sigma)\subset\mathbb{R}^{n+1}$, a solution to an extrinsic curvature flow, or $f$-flow for short, with initial data $\Sigma_0$ corresponds to a $1$-parameter family of immersions $F:\Sigma\times\mathbb{R}\to\mathbb{R}^{n+1}$, that solves
\begin{align}\label{f-flow}
\begin{cases}
\dfrac{\partial F}{\partial t}(x,t)=f(\lambda(x,t))\nu(x,t),\:(x,t)\mbox{ in }\Sigma\times(0,T),
\\
F(x,0)=F_0(x),
\end{cases}    
\end{align}
where $\nu(x,t)$ is the inward unit normal vector of the hypersurface $\Sigma_t=F(\Sigma,t)$ in $\mathbb{R}^{n+1}$, $\lambda(x,t)=(\lambda_1(x,t),\ldots,\lambda_n(x,t))$ are the principal curvatures of $\Sigma_t$ with respect to to $\nu(x,t)$. Here, $f:\Gamma\to \mathbb{R}$ is a smooth function of the principal curvatures of $\Sigma_t$ with the following properties: 
\begin{enumerate}[label=\alph*)]
\item\label{a} $\Gamma\subset\mathbb{R}^{n}$ is an open symmetric cone that contains the positive cone $\Gamma_+ \coloneqq \left\{\lambda\in\mathbb{R}^n: \lambda_i>0\right\}$.
\item\label{b} $f$ is positive and symmetric, i.e.: $f(\sigma(\lambda_1),\ldots,\sigma(\lambda_n))=f(\lambda_1,\ldots,\lambda_n)$ for every permutation $\sigma\in S_n$.
\item\label{c} $f$ is strictly increasing in each variable, i.e.: $\dfrac{\partial f}{\partial \lambda_i}(\lambda)>0$ holds for every $\lambda\in\Gamma$ and $i=1,\ldots,n$. 
			\item\label{d} $f$ is $1$-homogeneous, i.e.: $ f(c\lambda)=cf(\lambda)$ for every $c>0$. 
 \item\label{e} $f$ vanishes at boundary of $\Gamma$, i.e.: there exist a continous function $\tilde{f}:\overline{\Gamma}\to\mathbb{R}$ such that $\restr{\tilde{f}}{\Gamma}=f$ and $\restr{\tilde{f}}{\partial \Gamma}=0$.  
\end{enumerate}

Given a speed function $f$, a translating solution to Equation \eqref{f-flow}, or a \emph{f-translator} for short, is a solution of the form 
\begin{align*}
 F(x,t)=F_0(x)+e_{n+1}t.
\end{align*}
(up to tangential reparametrizations.) Note that by hypotheses \eqref{b} and \eqref{c}, the $f$-flow is invariant under the isometries of ambient space and parabolic rescalings. Thus there is no loss of generality in fixing the translation direction  to be $e_{n+1}=(0,\ldots,0,1)\in\mathbb{R}^{n+1}$. 
\newline

Importantly, translating solutions can be studied by the parabolic and by the elliptic PDE points of view, since every time slice $\Sigma_t$ satisfies the Equation 
\begin{align}\label{f-trans}
    f(\lambda)=-\pI{\nu,e_{n+1}},
\end{align}
recall that $\nu$ is the inward pointing unit normal of $\Sigma_t$ in $\mathbb{R}^{n+1}$. In fact, in local coordinates $\Sigma_t$ can be seen as a graph of a function for which Equation \eqref{f-trans} correspond to a nonlinear elliptic PDE (quasilinear when $f=H$ and fully nonlinear when $f\neq H$.). Moreover, from the parabolic point of view, $f$-translators without boundary are examples of noncompact \emph{eternal solutions} of Equation \eqref{f-flow}, i.e.: solutions that are defined for all $t\in(-\infty,\infty)$, see \cite{jtsmaximumprinciple} for details.
\newline

It is worth mentioning that $f$-translators have been widely studied when $f=H$, see for instance \cite{hoffman2021notes} for a complete survey about $H$-translators, model of singularities, and minimal surfaces theory. In addition, the reader will be referred to \cite{urbas} for existence and properties of $\sqrt[\alpha]{S_n}$-translators for $\alpha>0$.
\newline

In a different work, discussed in \cite{rengaswami2021rotationally}, the first author explored the existence, uniqueness, regularity, and asymptotic geometry of "bowl"-type solutions. These solutions are constructed for fairly general speed functions that are $\alpha$-homogeneous with $\alpha>0$. To be precise, a ``bowl''-type solution of \eqref{f-flow} is a complete, strictly convex, rotationally symmetric\footnote{This solution is unique among strictly convex rotationally symmetric translating graphs.} translating graph in $\mathbb{R}^{n+1}$ which may be defined in a ball of finite radius or all of $\mathbb{R}^n$.
\newline

This dichotomy is characterized by the value of $f(0,1,\ldots,1)$ and the asymptotic behavior of the implicit solution of $f(x,y,\ldots,y)=1$ as $y\to \infty$, i.e. the behavior of the speed function near the boundary of the positive cone. In addition, when $f(0,1,\ldots,1)>0$, the ``bowl''-type solution is always entire and behaves like a paraboloid at infinity, i.e:
\begin{align*}
\dfrac{|x|^2}{2f(0,1,\ldots,1)}+o(|x|^2),\mbox{ as }|x|\to\infty.  
\end{align*}
It is an interesting question what the lower order terms are. When $f$  is the mean curvature, the authors in \cite{CSS}  showed that the bowl soliton is smoothly asymptotic to 
\begin{equation}
    \frac{|x|^2}{2(n-1)}-\ln(|x|)+O(|x|^{-1}), \mbox{ as }|x|\to\infty.
\end{equation}
We extend the above result to a large class of speeds which we define below. For convenience, we denote $\e \coloneqq (1,\ldots,1)\in \R^{n-1}$, the $(n-1)-$tuple of 1's.

\begin{theorem}\label{T1}
 Assume that $f:\Gamma\to\mathbb[0,\infty)$ satisfies properties \ref{a}-\ref{d} and is nondegenerate. Then, the entire ``bowl''-type solution is smoothly asymptotics to 
    \begin{equation*}
        \dfrac{|x|^2}{2f(0,\e)}-\restr{\dfrac{\partial f}{\partial\lambda_1}}{(0,\e)} \ln (|x|) +O(|x|^{-1}),\mbox{ as }|x|\to\infty.
    \end{equation*}
\end{theorem}\,
Furthermore, by applying the same techniques employed in \cite{MHS}, we show that that the bowl-type soliton is essentially unique in the asymptotic class of $O(|x|^2)$ solutions via the following theorem:

\begin{theorem}\label{T3}
Let $\Sigma\subset\R^{n+1}$ be a strictly convex complete $f$-translator with a single end smoothly asymptotic to the ``bowl''-type solution. Then, if $f$ satisfies properties \ref{a}-\ref{e} and is nondegenerate, $\Sigma$ is the ``bowl''-type solution up to vertical translations.
\end{theorem}

To show that Theorem \ref{T1} does not apply to all speed functions, we discuss the degenerate speed function $f=\sqrt[n]{S_n}$ (the $n^{th}$ root of the Gauss curvature) and show that the bowl soliton does not have quadratic asymptotics for any $n\geq 2$. 
\begin{theorem}
 Let $r \coloneqq |x|$ be the Euclidean norm of an $n$-tuple. The ``bowl''-type solution for the speed function $f=\sqrt[n]{S_n}$ is smoothly asymptotic to 
\begin{align*}
    \begin{cases}
        \int\limits_0^{r}e^{\frac{s^2}{2}}ds+O\left(\int\limits_0^{r}\sqrt{e^\frac{s^2}{2}-1} ds\right),& \mbox{ for }n=2,
        \\
        \dfrac{r^4}{12}+O(r),& \mbox{ for }n=3,
        \\
        \dfrac{(n-2)^{\frac{n-1}{n-2}}}{2(n-1)^{\frac{1}{n-2}}}r^{\frac{2(n-1)}{n-2}}+O\left(r^{\frac{2}{n-2}}\right),&\mbox{ for }n\geq 4
    \end{cases},\mbox{ as }r\to\infty. 
\end{align*}
\end{theorem}
 
Then, we discuss a special kind of translator known as the \emph{wing-like} solution firstly studied in the context of mean curvature flows in \cite{CSS}.
\begin{theorem}
For every $R>0$, there exist a non-convex complete rotationally symmetric $f$-translator $W_R$ with respect to $x_{n+1}$-axis $f=\sqrt[k]{S_k}$ and $f=\dfrac{S_k}{S_{k-1}}$ such that:
\begin{enumerate}
    \item For $f=\sqrt[k]{S_k}$, we distinguish:
    \begin{enumerate}
        \item When $k$ is even: $W_R\setminus B_{R_1}(0)$ with $R_1>R$ posses two graphical branches $W_R^+,W_R^-:\mathbb{R}^n\setminus B_R(0)\to\mathbb{R}$ smoothly asymptotic to 
\begin{align*}
    W_R^\pm(x)=\pm\left(\dfrac{|x|^2}{2f(0,\e)}-\restr{\dfrac{\partial f}{\partial \lambda_1}}{(0,\e)}\ln(|x|)+O(|x|^{-1})+C^{\pm}\right),\mbox{ as } |x|\to\infty. 
\end{align*}    
      \item When $k$ is odd: $W_R$ posse a $\mathbb{S}^{n-1}$ boundary component and $W_R\setminus B_{R_1}(0)$ with $R_1>R$ is given by a vertical graph smoothly asymptotic to 
\begin{align*}
    W_R(x)=\dfrac{|x|^2}{2f(0,\e)}-\restr{\dfrac{\partial f}{\partial \lambda_1}}{(0,\e)}\ln(|x|)+O(|x|^{-1})+C,\mbox{ as } |x|\to\infty. 
\end{align*}    
    \end{enumerate}
    \item For $f=\dfrac{S_k}{S_{k-1}}$ we have that $W_R\setminus B_{R_1}(0)$ with $R_1>R$ posses two graphical branches $W_R^+,W_R^-$ such that
    \begin{align*}
    &W_R^+(x)=\dfrac{|x|^2}{2f(0,\e)}-\restr{\dfrac{\partial f}{\partial \lambda_1}}{(0,\e)}\ln(|x|)+O(|x|^{-1})+C^{+},\mbox{ as } |x|\to\infty, 
    \\
    &\lim\limits_{|x|\to\infty}|\nabla W_R^{-}(x)|=0.
    \end{align*}
\end{enumerate}
\end{theorem}

\begin{figure}
    \centering
    \begin{subfigure}[b]{0.32\textwidth}
        \includegraphics[scale=0.5]{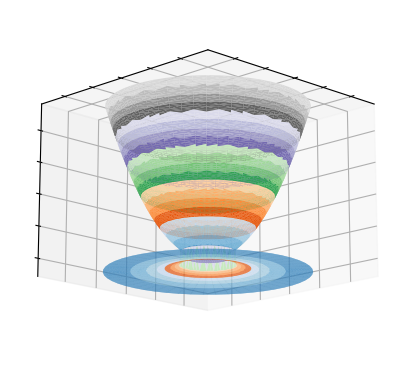}
        \caption{$S_k/S_{k-1}$}
        \label{fig:gull}
    \end{subfigure}
    \begin{subfigure}[b]{0.32\textwidth}
        \includegraphics[scale=0.5]{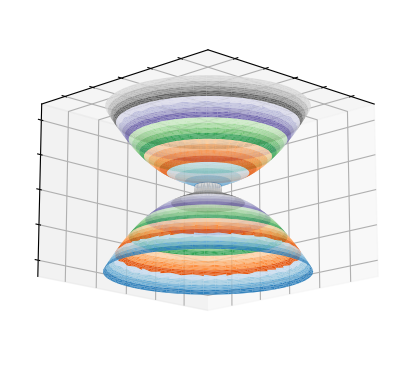}
        \caption{$\sqrt[k]{S_k}$ for $k$ even}
        \label{fig:tiger}
    \end{subfigure}   
    \begin{subfigure}[b]{0.32\textwidth}
        \includegraphics[scale=0.5]{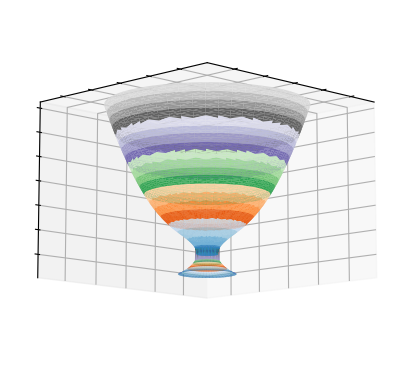}
        \caption{$\sqrt[k]{S_k}$ for $k$ odd}
        \label{fig:mouse}
    \end{subfigure}
    \caption{Winglike translators for various speed functions}
    \label{fig:constrainteqn}
\end{figure}

Finally, we provide some applications of the above theorems, namely an asymptotic growth estimate:
\begin{theorem}\label{T4}
    Let $\Sigma=\set{(x,u(x)):x\in\mathbb{R}^n}$ be an entire convex translating solution of \eqref{f-flow} such that $f$ is nondegenerate. Assume furhter, that there exist  $a,b,C_1,C_2,R>0$ such that 
    \begin{align}
    C_1|x|^a\leq u(x)\leq C_2|x|^b,\mbox{ for }|x|\geq R,
    \end{align}
    then, $a\leq 2\leq b$. In addition, if $a=b=2$, then $u(x)$ agrees with the ``bowl''-type solution up to vertical translations. 
\end{theorem}

The organization of this article goes as follows: In Section \ref{sec:Prelim}, we discuss the preliminaries of the differential geometry of axially symmetric translators and the ODE theory needed for the analysis of the ODE that the translator solves. In Section \ref{Asym}, we derive the asymptotic expansion up to $o(|x|^{-2}$) of the slope field of the translator. We use this estimate in Section \ref{sec: Uniqueness} to prove the uniqueness result in Theorem \ref{T3}. Section \ref{Sn} discusses the degenerate example $\sqrt[n]{S_n}.$ Section \ref{wing-like sec} concerns the discussion of winglike translators. Finally, in Section \ref{sec:Application} we give a proof of Theorems \ref{T4}.
\newline

\textbf{Acknowledgment:} 
We would like to acknowledge the support of our advisors Drs. Mat Langford and Theodora Bourni, and Franciso Martín, Miguel Sanchez and Mariel Saez, both in terms of technical discussions as well as travel funding.
The second author was partially supported by the project
	PID2020-116126GB-I00 funded by MCIN/ AEI /10.13039/501100011033,
	by the project PY20-01391 (PAIDI 2020) funded by Junta de Andaluc\'{\i}a
	FEDER and by the framework of IMAG-Mar \'{\i}a de Maeztu grant CEX2020-
	001105-M funded by MCIN/AEI/ 10.13039/50110001103.

\section{Preliminaries}\label{sec:Prelim}

\subsection{The rotational translator ODE} \label{sec:TODE}

For a real-valued $C^2$-function $u$ of a single real variable we consider $\Sigma$ to be the graph of $y=u(r)$, where $r=|x|$ and $x\in\R^n$. 
\\

Then, the inward unit normal of $\Sigma$ at a point $(x,u(r))$ is given by
\begin{align*}
    \Vec{N}=\left(\frac{u'}{\sqrt{1+u'^2}}\frac{x}{r},\frac{-1}{\sqrt{1+u'^2}}\right) \in \R^n \times \R
\end{align*}
Moreover, the principal curvatures of $\Sigma$ are given by
\begin{align}\label{Principal curvatures}
    \lambda_1=\frac{u''}{(1+u'^2)^{3/2}}\mbox{ and } \lambda_i=\frac{u'}{r\sqrt{1+u'^2}},\mbox{ for } i=2,\ldots,n.
\end{align}
\begin{definition}\label{Nondegenerate}
    A speed function is a function $f:\Gamma\to\R$ that satisfies properties \ref{a}-\ref{d}. A speed function $f$ is said to be \emph{nondegenerate} if $f(0,\e)>0$, where $\e=(1,\ldots,1)\in\R^{n-1}$.
\end{definition}
\begin{remark}
We emphasize that being \emph{nondegenerate} is equivalent to requiring that the cylinder $S^{n-1}\times \R$ is not a stationary solution to the $f$-flow \eqref{f-flow}.
\end{remark}
\begin{example}
The class of speed functions $f:\Gamma\to\mathbb{R}$ that additionally satisfies Property \ref{e} is vast and includes:
\begin{itemize}
    \item The mean curvature $H=\lambda_1+\ldots+\lambda_n$, supported in $\Gamma_1=\set{\lambda\in\mathbb{R}^n:H>0}$.
    \item The $k$-th roots of the symmetric elemental polynomial $\sqrt[k]{S_k}$, where
    \begin{align*}
        S_k(\lambda)=\sum_{1\leq i_1<\ldots<i_k\leq n}\lambda_{i_1}\ldots\lambda_{i_k},
    \end{align*}
    supported in the g$\accentset{\circ}{a}$rdin cone $\Gamma_k:=\set{\lambda\in\mathbb{R}^n:S_l(\lambda)>0,\: l=1,\ldots, k}$. 
    \item The inverse of the $k$-th harmonic sum $\left(\sum\limits_{1\leq i_1<\ldots<i_k\leq n}\dfrac{1}{\lambda_{i_1}+\ldots+\lambda_{i_k}}\right)^{-1}$ supported in $\Gamma=\set{\lambda\in\mathbb{R}^n: \lambda_1+\ldots+\lambda_k>0}$, where $\lambda_1\leq\ldots\leq\lambda_n$.
    \item Any $1$-homogeneous symmetric combination of the above functions. 
\end{itemize}
It is worth to mention that by removing hypothesis \ref{e}, the Hessian quotients functions $Q_{k,l}=\left(\dfrac{S_k}{S_l}\right)^{\frac{1}{k-l}}$ supported in $\Gamma_k$ can be included in this class of functions.
\end{example}
\begin{remark}\label{Rem 2}
Due to axial symmetry of $\Sigma$, $f$ only depends on two variables (because there are only two distinct principal curvatures), and hence we sometimes use $f(x,y)$ instead of $f(x,y\e).$    
\end{remark}

Therefore, by Remark \ref{Rem 2} equation 
\begin{align*}
    f(\lambda)=\pI{\nu,e_{n+1}}
\end{align*}
for a rotationally symmetric graph $(x, u(r))$ with respect to $x_{n+1}$-axis is given by 
\begin{align*} 
    f\left(\frac{u''}{(1+u'^2)^{3/2}},\frac{u'}{r\sqrt{1+u'^2}}\right)=\frac{1}{\sqrt{1+u'^2}}\,.
\end{align*}
We can reduced the above equation to a first order ODE by setting $v=u'$, and with the $1$-homogeneity of $f$ we may write it by 
\begin{equation}\label{f-trans ODE f=1}
    f\left(\frac{v'}{1+v^2},\frac{v}{r}\right)=1.
\end{equation}
\begin{remark} \label{def:g}
   Geometrically, $v$ is the gradient of the profile curve $(r,u(r))\in\R^2$, and the solution is unique up to vertical translations. 
\end{remark}

In addition, since $f$ is strictly monotone in each argument, we may apply the Implicit Function theorem to obtain a unique function $x=g(y,z)$ in the sense that 
\begin{align}\label{wing-like f}
    f(g(y,z),y )=z.
\end{align}
Note that in Eq. \eqref{f-trans ODE f=1}, we have $z=1$, and in this case, we will suppress the second argument and refer to $g(y,1)$ as simply $g(y)$. We will return to using $g(y,z)$ in Section \ref{wing-like sec}.
\\

Moreover, since $f$ is $C^1$ and has a non-singular derivative w.r.t. $x$, we have that $g$ is of class $C^1$ as well.  
\newline

Consequently, the slope function $v$ of a rotationally symmetric $f$-translators satisfies
\begin{align}\label{f-trans ODE g}
    \begin{cases}
    v'(r)=\left(1+v^2(r)\right) g \left (\dfrac{v(r)}{r}\right), r\geq 0,
    \\
    v(0)=0.
    \end{cases}, 
\end{align}
and we can recover $u(r)$ from $v(r)$ via an integration procedure. 
\newline

\begin{remark}
It is important to mention that even though Equation \eqref{f-trans ODE g} appears to have a singularity for the  initial condition $v(0)=0$, this is only a coordinate singularity.
\\
We refer the interested reader to \cite{rengaswami2021rotationally} for a study of general $\alpha-$homogeneous speeds, where  the questions of existence, regularity, uniqueness, and convexity of this solution are addressed. A complete classification of speeds based on whether the resulting solution is asymptotically cylindrical is also presented therein.    
\end{remark}

\begin{example}
The following examples are the expressions of the function $f(x,y)$ and $g(y)$ for the following speed functions:
    \begin{enumerate}
    \item The mean curvature, $H$:
     \begin{align*}
      f(x,y)=x+(n-1)y\mbox{ and }g(y)=1-(n-1)y.   
     \end{align*} 
     \item The $k$th-root of the symmetric elemental polynomials, $\sqrt[k]{S_k}$:
\begin{align*}
        f(x,y)=\sqrt[k]{\binom{n-1}{k}y^k +\binom{n-1}{k-1}xy^{k-1}}\mbox{ and }g(y)=\binom{n-1}{k-1}^{-1}y^{1-k}-\frac{n-k}{k}y.
\end{align*}
    \item The quotients of the symmetric elemental polynomials, $Q_{k+1,k}=\dfrac{S_{k+1}}{S_k}$: 
    \begin{align*}
    f(x,y)=\dfrac{\binom{n-1}{k}xy^{k}+\binom{n-1}{k+1}y^{k+1}}{\binom{n-1}{k-1}xy^{k-1}+\binom{n-1}{k}y^{k}}\mbox{ and }g(y)=\dfrac{n-k}{k+1}y\dfrac{(k+1)-(n-k-1)y}{(n-k)y-k}.
    \end{align*}
\end{enumerate}
\end{example}

\subsection{Differentiability properties of \emph{nondegenerate} $f$}
We are interested in speed function $f$ that satisfies $f(0,1)>0$, in this section and beyond, we normalize $f$ so that $f(0,1)=1$. 
\\

In addition, since $f(x,y)$ is $1$-homogeneous, we have the identity
\[f(x,y)=f_x x+f_y y.\]
where $f_x=\frac{\partial f}{\partial x}$ and $f_y=\frac{\partial f}{\partial x}$.
\\
We note that the partial derivatives $f_x$ and $f_y$ are $0$-homogeneous functions, and  by our normalization, we have 
\begin{align}\label{f_y(0,1)}
    f_y(0,1)=1
\end{align}
Now, we outline an estimation trick that we will use repeatedly in this paper. 
\\
Due to $1$-homogeneity of $f$, we have $ f(x,y)=yf(x/y,1)$:
\begin{itemize}
    \item When $f$ is $C^1$, using the mean value theorem, we may write
\begin{align}\label{Taylor1}
    f(x,y)=y(1+f_x(\xi,1))\dfrac{x}{y}=y+f_x(\xi,1)x
\end{align}
for some $0\leq \xi \leq \dfrac{x}{y}$.
\item When $f$ is $C^2$,  Taylor's remainder theorem yields
\begin{equation}\label{Taylor2}
    f(x,y)=y\left(1+f_x(0,1)\dfrac{x}{y}+\frac{1}{2}f_{xx}(\xi,1)\dfrac{x^2}{y^2}\right)=y+f_x(0,1)x+\frac{1}{2y}f_{xx}(\xi,1)x^2
\end{equation}
\end{itemize}

Therefore, by the normalization of $f$, we will assume in most calculations in this paper that 
\begin{align*}
    0<x\leq y\Leftrightarrow 0 < \xi \leq 1,
\end{align*}
and consequently, $|f_x(\xi,1)|,|f_{xx}(\xi,1)|$ are bounded by compactness of $[0,1]$ and the continuity of these functions.
%When $0 \leq x \leq y$  and $x,y$ are positive, we have $0\leq \xi \leq 1$ so that
%\begin{equation}
   % f(x,y)=y+f_x(0,1)x+O(x^2).
%\end{equation}

\subsection{Differentiability properties of $g$}
Recall that there exist a unique $\mathcal{C}^1$ function $g$ that satisfies $f(g(y),y)=1$. 
\\

Then, due to our normalization $f(0,1)=1$, we have $g(1)=0$. Moreover, by chain rule,
\[f_x(g(y),y)g_y(y)+f_y(g(y),y)=0\]
which, by suppressing the arguments, we get
\[f_xg_y+f_y=0.\]
Differentiating once more, we get
\[f_{xx}g_y^2+2f_{xy}g_y+f_{yy}+f_xg_{yy}=0.\]

Next, we note that whenever $f_x \neq 0$ it holds
\[g_y=-\frac{f_y}{f_x}.\]
In particular, $g_y \leq 0$, indicating that $g(y)$ is decreasing.

\begin{remark}
In particular, when $f_{xx},f_{xy},f_{yy}$ are defined at $(x,y)=(g(y),y)$, we have
\[g_{yy}=-\frac{f_{xx}g_y^2+2f_{xy}g_y+f_{yy}}{f_x}.\]
Therefore, it is interesting that $g$ will be convex when $f$ is concave and vice versa. However, since our results do not rely on the convexity properties of $f$, we won't be using this fact.    
\end{remark}

\subsection{ODE Theory}\label{subsec:ODE}
Throughout the paper we will use the technique of super-solutions and sub-solutions for an ODE of the form
\begin{equation} \label{generalode}
  x'(t)=f(t,x(t)).
\end{equation}
We refer the reader to \cite{teschl2012ordinary} for the following classic ODE results. 
\\

A differentiable function $x_+(t)$ satisfying
\[x_+'(t)>f(t,x_+(t))\] is called a \emph{super-solution} to (\ref{generalode}). Similarly, a differentiable function $x_-(t)$ satisfying
\[x_-'(t)<f(t,x_-(t))\] is called a \emph{sub-solution} to \eqref{generalode}.

\newtheorem{Lemma}{Lemma}
\begin{lemma}
Let $x_+(t)$, $x_-(t)$ be super, sub-solutions of the differential equation $x'=f(t,x)$ on $[t_0,T)$ respectively. For every solution $x(t)$ on $[t_0,T)$ we have 
\[x(t) < x_+(t), t \in [t_0,T) \text{ whenever }  x(t_0) \leq x_+(t_0)\]
respectively
\[x_-(t) < x(t), t \in [t_0,T) \text{ whenever } x_-(t_0) \leq x(t_0)\]
\end{lemma}

\begin{remark}
    If one replaces strong inequality by weak inequality in the definitions of sub and super-solutions, one gets weak inequalities instead of strong ones in the above lemma.
\end{remark}

\section{Asymptotic of bowl-type solutions}\label{Asym}
%\begin{proposition}
    %Suppose $f_y(0,1)>0$. Then $\psi(r)\to0$ as $r \to \infty$  
%\end{proposition}
%\begin{proof}
%This is essentially a generaliation of the argument of Clutterbuck, Schnurer, Schulze. Suppose $\psi(r) \leq -\epsilon$. Since $\psi$ is sublinear, we may also assume that eventually, $r+\psi \geq r/2$

%Then
%\begin{align*}
    %\psi'&=v'-1\\
    %&= %(1+v^2)g(1+\psi/r,1)-1\\ 
    %&\geq (1+r^2/4)g_y(1,1)(-\epsilon/r)-1\\
    %&\geq c\\
    %&>0
%\end{align*}
%for large $r$ provided $g_y(0,1)<0$, or equivalently $f_y(0,1)>0$.
%\end{proof}
Recall from the introduction that a bowl-type soliton is a complete strictly convex smooth solution of \eqref{f-trans ODE g} with $u'=v$, which is defined in $[0,R)$ where $R\in\set{\frac{1}{f(1,1)},\infty}$. 
\\

Furthermore, by the \emph{nondegenerate} property together with the normalization $f(0,1)=1$, we have that $R=\infty$ and  
\begin{align*}
u(r)=r^2+o(r^2),\mbox{ as }r\to\infty.  
\end{align*}

In this section we will prove Theorem \ref{T1} which state:
\begin{theorem}
 Assume that $f$ is a normalized \emph{nondegenerte} speed function. Then, the corresponding bowl-type soliton for this flow has the asymptotics
    \[u(r)=\frac{r^2}{2}-c\log r+o(r^{-1}),\mbox{ as }r\to\infty,\]
    where $c=f_x(0,1)$.
\end{theorem}
We will prove the above theorem by showing that $v$ satisfies
\begin{align*}
    v(r) = r-\frac{c}{r}+o(r^{-2}),\mbox{ as }r\to\infty,
\end{align*}
whence the claim of the theorem follows at once by integration.

% One might wonder how one arrives at these coefficients, and why the asymptotics of $v$ have no $r^0$ or $r^{-2}$ terms. To see this intuitively, note that the equation \ref{eqn:bowlode} has the symmetry of an odd function: $\Bar{v}(r) \coloneqq -v(-r)$ also satisfies the same ODE.
\begin{remark}
    Intuitively,  a differentiable solution $v$ to \eqref{f-trans ODE g} has the symmetries of an odd function, and hence one does not expect even powers in the asymptotic expansion, and indeed one can formally plug in a Laurent series for $v$, and without much effort, one can determine the coefficients of the asymptotic expansion. In what follows, we provide rigorous proofs for the correctness of these coefficients.
\end{remark}
The method that we will follow consists in a bootstrapping approach to progressively refine the asymptotic terms. For this, we will use the ODE's theory of sub- and super-solutions denoted by $w_{\pm,\varepsilon}(r)$, where the $\varepsilon$ will be used in asymptotic  little-oh notation.  
\\

On the other hand, from Section \ref{subsec:ODE}, we recall that $w$ is a sub-solution (super-solution resp.) to Eq. \eqref{f-trans ODE g} if 
\[w' \leq  (1+w^2)g(w/r),\:(\geq \mbox{ resp.}).\]
However, in some cases it may be more convenient to check the following equivalent condition
\begin{align}\label{eqn: fsubsol} 
    f\left(\frac{w'}{1+w^2},\frac{w}{r}\right) \leq 1,\:(\geq \mbox{ resp.}).
\end{align}

\begin{proposition}\label{o(r)}
The functions 
\begin{align*}
 w_+(r)=r  \mbox{ and }w_{-,\varepsilon}(r)=(1-\varepsilon)r 
\end{align*}
satisfy the following properties:
    \begin{enumerate}
        \item $w_+(r)$ is a super-solution to Eq. \eqref{f-trans ODE g} $r>0$. 
        \item For every $\epsilon\in (0,1)$, $w_{-,\varepsilon}(r)$ is a sub-solution to Eq. \eqref{f-trans ODE g} for sufficiently large $r$. Moreover, given any $r_0>0$, there exists $r_1>r_0$ such that  $v(r_1)\geq r_1$. Thus, $v \geq (1-\varepsilon)r$ for sufficiently large $r$.
    \end{enumerate}
\end{proposition}

\begin{proof}
The proofs are all direct computations.
    \begin{enumerate}
        \item Firstly, we note that $f(x,y)$ is increasing in each variable, then it holds
        \begin{align*}
            f\left(\frac{w_+'}{1+w_+^2},\frac{w_+}{r}\right)=f\left(\frac{1}{1+r^2},1\right)
            \geq f\left(0,1\right)
            =1.
        \end{align*}
        Therefore,  $w_+(r)$ is a super-solution to Eq. \eqref{f-trans ODE g} $r>0$.

        \item Next, by evaluating $w_{-,\varepsilon}$ in \eqref{f-trans ODE f=1} and taking the limit as $r\to\infty$, we see that
        \begin{align*}
            \displaystyle \lim_{r\to\infty} f\left(\frac{w_{-,\varepsilon}'}{1+w_{-,\varepsilon}^2},\frac{w_{-,\varepsilon}}{r}\right)&= \lim_{r\to\infty} f\left(\frac{1-\epsilon}{1+((1-\varepsilon)r)^2},1-\varepsilon\right)
            \\
            &=f\left(0,1-\varepsilon\right)
            \\
            &<f(0,1)=1.
        \end{align*}
        Therefore, the claim is true for sufficiently large $r$.
      \\
      
        On the other hand, we will prove the following part  by contradiction. Let $r_0>0$ and assume that $v(r)<(1-\epsilon)r$ for all $r>r_0$. Then, since $g(y)$ is decreasing in $y$, it follows that
        \begin{align*}
            v'&=(1+v^2)g\left(\frac{v}{r}\right)\\
              &> (1+v^2)g\left(1-\varepsilon\right)\\
              &=C(1+v^2)
        \end{align*}
        for $r>r_0$ and some $C>0$. Consequently, this inequality implies that $v$ blows up at some finite $r_1>r_0$ contradicting that $v(r)$ exist for all $r\geq 0$.    \end{enumerate}
\end{proof}

\begin{remark}
   Proposition \ref{o(r)} implies that the solution $v(r)=r+o(r)$ as $r\to\infty$.
\end{remark}

\begin{proposition} \label{o(1)}
The function $w_{-,\varepsilon}(r)=r-\varepsilon$ satisfies:
    \begin{enumerate}
        \item For every $\varepsilon\in (0,1)$, $w_{-,\varepsilon}(r)$ is a sub-solution to Eq. \eqref{f-trans ODE f=1} for sufficiently large $r$.
        \item Given any $r_0>0$, there exists $r_1>r_0$ such that  $v(r_1)\geq r_1$. Thus, $v \geq r-\epsilon$ for sufficiently large $r$.
    \end{enumerate}
\end{proposition}

\begin{proof}
    \begin{enumerate}

        \item We will show that $w_{-\varepsilon}$ verifies Eq. \eqref{eqn: fsubsol}. Indeed, by Eq. \eqref{Taylor1}, we may write
        \begin{align*}
         f\left(\frac{w_{-,\varepsilon}'}{1+w_{-,\varepsilon}^2},\frac{w_{-,\varepsilon}}{r}\right)&=  f\left(\frac{1}{1+(r-\varepsilon)^2},1-\frac{\epsilon}{r}\right)\\
            &=1-\frac{\varepsilon}{r}+\frac{f_x(\xi,1)}{1+(r-\varepsilon)^2}.
         \end{align*} 
         Then, since 
         \begin{align*}
         0\leq \xi \leq \dfrac{1}{\left(1-\dfrac{\varepsilon}{r}\right)(1+(r-\varepsilon)^2)}<1    
         \end{align*}
        for $r$ sufficiently large, we obtain
           \begin{align*}
            f\left(\frac{w_{-,\varepsilon}'}{1+w_{-,\varepsilon}^2},\frac{w_{-,\varepsilon}}{r}\right)&=1-\frac{\varepsilon}{r}+o(r^{-1})<1.
        \end{align*}
        
        for sufficiently large $r$.

        \item Next, by arguing by contradiction, we fix $r_0>0$ and assume that $v(r)<r-\varepsilon$ holds for all $r>r_0$ and $\varepsilon\in (0,1)$. 
        \newline
        Then, by part 2 of Proposition \ref{o(r)}, we have that $v\geq \dfrac{r}{2}$ for sufficiently large $r$ and
        \begin{align*}
            v'&=(1+v^2)g\left(\frac{v}{r}\right)\\
              &> \left(1+\dfrac{r^2}{4}\right)g\left(1-\dfrac{\varepsilon}{r}\right).
        \end{align*}
        Finally, recall that $g(1)=0$. Then, by the mean value theorem, we have 
        \begin{align*}
              v'&\geq \left(1+\dfrac{r^2}{4}\right)\left(g(1)-g_y\left(\xi\right)\dfrac{\varepsilon}{r}\right),
        \end{align*}      
         for some $\xi\in \left(1-\dfrac{\varepsilon}{r},1\right)$. Consequently, since $g$ is decreasing, there is $C>0$ depending on $\varepsilon$ such that     
        \begin{align*}
             v' &\geq Cr,\mbox{ for all }r>r_0.
        \end{align*}
       However, this fact contradicts $v(r)<r-\varepsilon$ for all $r>r_0$. 
        \end{enumerate}
\end{proof}

\begin{remark}
Proposition \ref{o(1)} implies that $v(r)=r+o(1)$ as $r\to\infty$. 
\end{remark}

\begin{proposition} \label{o(r^-1)}
Let $c=f_x(0,1)$ and consider 
\begin{align*}
    w_{+,\varepsilon}(r)=r-\frac{c-\varepsilon}{r}\mbox{ and }w_{-,\varepsilon}(r)=r-\frac{c+\epsilon}{r}.
\end{align*}
    \begin{enumerate}
        \item For every $\varepsilon \in(0,c)$, $w_{+,\varepsilon}(r)$ is a super-solution to Eq. \eqref{f-trans ODE f=1} for sufficiently large $r$. In addition, given any $r_0>0$, there exists $r_1>r_0$ such that  $v(r_1)\leq r_1-\frac{c-\varepsilon}{r_1}$. Thus, $v \leq r-\frac{c-\varepsilon}{r}$ for sufficiently large $r$.
        \item For every $\epsilon>0$, $w_{-,\varepsilon}(r)$ is a sub-solution to Eq. \eqref{f-trans ODE f=1} for sufficiently large $r$. Moreover, given any $r_0>0$, there exists $r_1>r_0$ such that  $v(r_1)\geq r_1$. Thus, $v \geq r-\frac{c+\varepsilon}{r}$ for sufficiently large $r$.
    \end{enumerate}
\end{proposition}
\begin{proof}
    \begin{enumerate} 
        \item Taking sufficiently large $r$ as in part 2 of Proposition \ref{o(1)}, but using Eq. \ref{Taylor2} instead, we have
        \begin{align*}
         f\left(\frac{w_{+,\varepsilon}'}{1+w_{+,\varepsilon}^2},\frac{w_{+,\varepsilon}}{r}\right)&= f\left(\frac{1+\frac{c-\varepsilon}{r^2}}{1+(r-\frac{c-\varepsilon}{r})^2},1-\frac{c-\varepsilon}{r^2}\right)\\
         &\geq f\left(\frac{1}{1+r^2},1-\frac{c-\varepsilon}{r^2}\right)\\
            &\geq 1-\frac{c-\varepsilon}{r^2}+\frac{f_x(0,1)}{1+r^2}+\frac{1}{2(1-\frac{c-\varepsilon}{r^2})}f_{xx}(\xi,1)\frac{1}{(1+r^2)^2}\\
            &= 1+\frac{\varepsilon}{r^2}-\frac{c}{r^2(1+r^2)}+\frac{1}{2(1-\frac{c-\varepsilon}{r^2})}\frac{f_{xx}(\xi,1)}{(1+r^2)^2}\\
            &\geq 1+\frac{\varepsilon}{r^2}-O(r^{-4})\\
            &>1,\mbox{ for sufficiently large } r.
        \end{align*}
For the next part, let $r_0>0$ and we assume 
\begin{align*}
    v(r)>r-\frac{c-\epsilon}{r},\mbox{ for all }r>r_0.
\end{align*}
Then, by Proposition \ref{o(r)}, we have
    \begin{align*}
            v'&=(1+v^2)g\left(\frac{v}{r}\right)
               \\
              &\leq (1+r^2)g\left(1-\frac{c-\varepsilon}{r^2}\right),
    \end{align*}
since $g$ is decreasing. 
\\
We note that the second-order expansion of $g$ allows us to write 
    \begin{align*}    
              v&\leq(1+r^2)\left(g(1)-g_y(1)\frac{c-\epsilon}{r^2}+\frac{g_{yy}(\xi)}{2}\dfrac{c^2}{r^4}\right).
    \end{align*}          
Recall \eqref{f_y(0,1)},  $c=f_x(0,1)$ and $g(1)=0$, then we have $g_y(1)=-c^{-1}$ and
    \begin{align*}          
              v'&\leq(1+r^2)\left(\dfrac{c-\varepsilon}{cr^2}
              +\dfrac{g_{yy}(\xi)}{2}\dfrac{c^2}{r^4}\right)
              \\
              &=(1+r^2)\dfrac{c-\varepsilon}{cr^2}
              +O(r^{-2})
              \\
              &\leq \left(1-\dfrac{\varepsilon}{c}\right)+O(r^{-2}).
        \end{align*}
Therefore, we have obtained $v$ grows no more rapidly than $\left(1-\dfrac{\epsilon}{c}\right)r$, contradicting that $v(r)>r-\dfrac{c-\varepsilon}{r}$ for large enough $r>r_0$.

  \item Arguing as in the previous part, for sufficiently large $r$, we have
        \begin{align*}
         f\left(\frac{w_{-,\varepsilon}'}{1+w_{-,\varepsilon}^2},\frac{w_{-,\varepsilon}}{r}\right)&= f\left(\frac{1-\frac{c+\varepsilon}{r^2}}{1+(r-\frac{c+\varepsilon}{r})^2},1-\frac{c+\varepsilon}{r^2}\right)\\
         &\leq f\left(\frac{1}{1+r^2},1-\frac{c-\epsilon}{r^2}\right)
         \\
        &= 1-\frac{c+\varepsilon}{r^2}+\frac{f_x(0,1)}{1+r^2}+\frac{f_{xx}(\xi,1)}{2\left(1-\frac{c+\varepsilon}{r^2}\right)}\frac{1}{(1+r^2)^2}
            \\
            &= 1-\frac{\varepsilon}{r^2}-\frac{c}{r^2(1+r^2)}+O(r^{-4})
            \\
            &= 1-\frac{\varepsilon}{r^2}+O(r^{-4})
            \\
            &<1.
        \end{align*}

    For the next part, we fix $r_0>0$ and assume
    \begin{align*}
    v(r)<r-\frac{c+\varepsilon}{r},\mbox{ for all }r>r_0.
    \end{align*}
    Then, by part (2) of Proposition \ref{o(1)} for sufficiently large $r$, we have $v \geq r-1$ and hence,
        \begin{align*}
            v'&=(1+v^2)g\left(\frac{v}{r}\right)
            \\
             &\geq (1+(r-1)^2)g\left(1-\frac{c+\epsilon}{r^2}\right)
              \\
               &=(1+(r-1)^2)\left(g(1)-g_y\left(1\right)\frac{c+\varepsilon}{r^2}+\dfrac{g_{yy}(\xi)}{2}\left(\frac{c+\varepsilon}{r^2}\right)^2\right)
                \\
              &\geq (1+(r-1)^2)\left(\frac{c+\varepsilon}{cr^2}-O(r^{-4})\right)
               \\
              &\geq 1+\frac{\varepsilon}{c}+O(r^{-2}).
        \end{align*}
      Therefore, $v$ grows faster than $\left(1+\dfrac{\varepsilon}{c}\right)r+O(r^{-1})$, contradicting $v(r)<r-\dfrac{c+\epsilon}{r}$ for large enough $r>r_0$.   
        \end{enumerate}
\end{proof}

\begin{remark}
    Proposition \ref{o(r^-1)} implies  $v(r)=r-\dfrac{c}{r}+o(r^{-1})$ as $r\to\infty$. 
\end{remark}

\begin{remark}
    The results of Proposition \ref{o(r^-1)} were obtained independently in \cite{cogo2023rotational}.
\end{remark}

\begin{proposition}\label{o(r^-2)}
Let $c=f_x(0,1)$ and consider 
\begin{align*}
    w_{+,\varepsilon}(r)=r-\frac{c}{r}+\dfrac{\varepsilon}{r^2}\mbox{ and }w_{-,\varepsilon}(r)=r-\frac{c}{r}-\dfrac{\varepsilon}{r^2}.
\end{align*}
    \begin{enumerate}
        \item For every $\varepsilon \in(0,c)$, $w_{+,\varepsilon}(r)$ is a super-solution to Eq. \eqref{f-trans ODE f=1} for sufficiently large $r$. In addition, given any $r_0>0$, there exists $r_1>r_0$ such that  $v(r_1)\leq r_1-\dfrac{c}{r_1}+\dfrac{\varepsilon}{r_1^2}$. Thus, $v \leq r-\dfrac{c}{r}+\dfrac{\varepsilon}{r^2}$ for sufficiently large $r$.
        \item For every $\varepsilon>0$, $w_{-,\varepsilon}(r)$ is a sub-solution to Eq. \eqref{f-trans ODE f=1} for sufficiently large $r$. Moreover, given any $r_0>0$, there exists $r_1>r_0$ such that  $v(r_1)\geq r_1-\dfrac{c}{r_1}-\dfrac{\varepsilon}{r_1^2}$. Thus, $v \geq r-\dfrac{c}{r}-\dfrac{\varepsilon}{r^2}$ for sufficiently large $r$.
    \end{enumerate}
\end{proposition}
\begin{proof}
    \begin{enumerate}

        \item Indeed, as in the previous proof for sufficiently large $r$, we have
        \begin{align*}
            f\left(\frac{w_{+,\varepsilon}'}{1+w_{+,\varepsilon}^2},\frac{w_{+,\varepsilon}}{r}\right)&=  f\left(\frac{1+\frac{c}{r^2}-\dfrac{2\varepsilon}{r^3}}{1+\left(r-\dfrac{c}{r}+\dfrac{\varepsilon}{r^3}\right)^2},1-\dfrac{c}{r^2}+\dfrac{\varepsilon}{r^3}\right)
            \\
            &\geq f\left(\dfrac{1+\dfrac{c}{r^2}-\dfrac{2\varepsilon}{r^3}}{1+r^2},1-\dfrac{c}{r^2}+\dfrac{\varepsilon}{r^3}\right)
            \\
            &\geq 1-\dfrac{c}{r^2}+\dfrac{\varepsilon}{r^3}+\dfrac{c}{1+r^2}\left(1+\dfrac{c}{r^2}-\dfrac{2\varepsilon}{r^3}\right)+\dfrac{f_{xx}(\xi,1)}{2}\dfrac{\left(1+\dfrac{c}{r^2}-\dfrac{2\varepsilon}{r^3}\right)^2}{\left(1+\dfrac{c}{r^2}-\dfrac{\varepsilon}{r^3}\right)(1+r^2)^2}
            \\
            &= 1+\frac{\varepsilon}{r^3}-O(r^{-4})
            \\
            &\geq 1. 
        \end{align*}

        %\item Taking sufficiently large $r$ as in part (4) of Proposition \ref{o(r)},
        %\begin{align*}
         %f\left(\frac{w_+^'}{1+w_+^2},\frac{w_+}{r}\right)&=  f\left(\frac{1+\frac{c}{r^2}-\frac{2\epsilon}{r^3}}{1+(r-\frac{c}{r}+\frac{\epsilon}{r^3})^2},1-\frac{c}{r^2}+\frac{\epsilon}{r^3}\right)\\
        % &\geq f\left(\frac{1+\frac{c}{r^2}-\frac{2\epsilon}{r^3}}{1+r^2},1-\frac{c}{r^2}+\frac{\epsilon}{r^3}\right)\\
        % &\geq 1-\frac{c}{r^2}+\frac{\epsilon}{r^3}+f_x(0,1)\left(\frac{1+\frac{c}{r^2}-\frac{2\epsilon}{r^3}}{1+r^2}\right)+\frac{1}{2(1-\frac{c}{r^2}+\frac{\epsilon}{r^3})}f_{xx}(\xi,1)\left(\frac{1+\frac{c}{r^2}-\frac{2\epsilon}{r^3}}{1+r^2}\right)^2\\
        %&= 1+\frac{\epsilon}{r^3}-\frac{c}{r^2}+\frac{c}{1+r^2}+O(r^{-4})\\
        %&= 1+\frac{\epsilon}{r^3}+O(r^{-4})\\
        %&>1
        %\end{align*}
        
        %for sufficiently large $r$.
        For the next part, fix $r_0>0$ and suppose $v(r)>r-\dfrac{c}{r}+\dfrac{\varepsilon}{r^2}$ for all $r>r_0$. Then, by recalling that $g(1)=0$, $g_y(1)=-c^{-1}$, we obtain
        \begin{align*}
            v'&=(1+v^2)g\left(\frac{v}{r}\right)
            \\
            &\leq (1+r^2)g\left(1-\frac{c}{r^2}+\frac{\varepsilon}{r^3}\right)
            \\
            &\leq (1+r)^2\left(g(1)-g_y(1)\left(\frac{c}{r^2}-\frac{\varepsilon}{r^3}\right)+\dfrac{g_{yy}(\xi)}{2}\left(\frac{c}{r^2}-\frac{\varepsilon}{r^3}\right)^2\right)
            \\
              &\leq 1-\frac{\varepsilon}{cr}+O(r^{-2}).
        \end{align*}
     However, this contradicts $v(r)>r-\dfrac{c}{r}+\dfrac{\varepsilon}{r^2}$ for all $r>r_0$.

    \item As in part one for sufficiently large $r$, we have
        \begin{align*}
         f\left(\frac{w_{-,\varepsilon}'}{1+w_{-,\varepsilon}^2},\frac{w_{-,\varepsilon}}{r}\right)&=  f\left(\frac{1+\dfrac{c}{r^2}+\dfrac{2\varepsilon}{r^3}}{1+\left(r-\dfrac{c}{r}-\dfrac{\varepsilon}{r^2}\right)^2},1-\frac{c}{r^2}-\frac{\varepsilon}{r^3}\right)\\
         &\leq 1-\frac{\varepsilon}{r^3}+O(r^{-4})\\
         &\leq 1.
        \end{align*}
    On the other hand, fix $r_0>0$ and we assume that $v(r)<r-\dfrac{c}{r}-\dfrac{\varepsilon}{r^2}$ for all $r>r_0$. Then, by Part 2 of Proposition \ref{o(r^-1)} for sufficiently large $r$, we have $v>r-\dfrac{2c}{r}$ and hence,
        \begin{align*}
            v'&= \left(1+v^2\right)g\left(\frac{v}{r}\right)
            \\
            v'&\geq \left(1+\left(r-\frac{2c}{r}\right)^2\right)g\left(1-\frac{c}{r^2}-\frac{\varepsilon}{r^3}\right)
            \\
            &\geq 1+\frac{\varepsilon}{cr}-O(r^{-2}).
        \end{align*}
    However, this contradics $v(r)<r-\dfrac{c}{r}-\dfrac{\varepsilon}{r^2}$ for large enough $r>r_0$.        
        \end{enumerate}
\end{proof}

\begin{remark}
In the above proposition we have proven Theorem \ref{T1}, $v(r)=r-\dfrac{c}{r}+o(r^{-2})$.
\end{remark}

\begin{remark}
By following this process further with the speed functions $f=Q_{k+1,k}=\dfrac{S_{k+1}}{S_k}$, we obtained the following expression for the ``bowl''-type solution
\begin{align*}
v(r)=r-\dfrac{c}{r}+\dfrac{n^2(n-k-1)(n-k-4)}{(k+1)^3(n-k)^2}\dfrac{1}{r^3}+O(r^{-5}).
\end{align*}
Note that when $k=0$, $Q_{k+1,k}=H$ and this expression coincide with the one obtained in \cite{CSS}. 

\end{remark}

\section{Uniqueness of entire solutions}\label{sec: Uniqueness}\,
With the fine asymptotic information about the bowl-type solitons that we have obtained so far, we can now demonstrate the uniqueness result for entire strictly convex solutions of \eqref{f-trans} that are smoothly asymptotic to the ``bowl''-type solution for nondegenerate speed $f$. 
\newline

To this end, we will recall a tangential principle for $f$-translators in $\mathbb{R}^{n+1}$ developed by the second author in \cite{jtsmaximumprinciple}.

\begin{theorem}\label{T2}
	Let $\Sigma_1,\Sigma_2\subset\mathbb{R}^{n+1}$ be two complete, embedded, connected  $f$-translators  such that
	\begin{enumerate}
		\item $f:\Gamma\to(0,\infty)$ satisfies properties \ref{a}-\ref{d}.
		\item $\Sigma_1$ is strictly convex, i.e: the principal curvatures $\lambda\in\Gamma_+$.
		\item $\Sigma_2$ is convex, i.e: the principal curvatures $\lambda\in\overline{\Gamma}_+$ .
	\end{enumerate}
	Then,
	\begin{enumerate}
		\item Assume that there exists an interior point $p\in \Sigma_1\cap \Sigma_2$ such that the tangent spaces coincide at $p$. If $\Sigma_1$ lies at one side of $\Sigma_2$, then both hypersurfaces coincide.
		
		\item Assume that the boundaries $\partial \Sigma_i$ lie in the same hyperplane $\Pi$ and the intersection of $\Sigma_i$ with $\Pi$ is transversal. If $\Sigma_1$ lies at one side of $\Sigma_2$ and there exist $p\in \partial \Sigma_1\cap \partial \Sigma_2$ such that the tangent spaces to $\Sigma_i$ and $\partial\Sigma_i$ coincide, then both hypersurfaces coincide. 
	\end{enumerate}
\end{theorem}

\begin{proof}
See \cite{jtsmaximumprinciple} Theorem 1.4 for a proof and an additional remark.  
\end{proof}

\begin{remark}\label{Tangency Principle}
We would like to highlight that the hypotheses of the Tangency Principle can be modified by the followings:
	\begin{enumerate}
 \item  $\Gamma\subset\R^n$ is the connected component of $\set{\lambda\in\R^{n}:\gamma(\lambda)>0}$ containig $\Gamma_+$ and it is a convex cone.
		\item $f:\Gamma\to \R$ is smooth and symmetric and satisfies properties \ref{c}-\ref{d}.
         \item The principal curvatures of $\Sigma_1$ lies in $\Gamma$ and the principal curvatures of $\Sigma_2$ lies in $\partial\Gamma\cap\set{\frac{\partial f}{\partial\lambda_i}>0}$.
	\end{enumerate}
The reason of this is that the proof is based on a convex combination argument of the local graphs near a tangency point of the hypersurfaces. An in particular, by shrinking the domain if it necessary, the principal curvatures of the convex combination is an admissible family for the functional $f(\lambda)-\pI{\nu,e_{n+1}}$, where a Hopf's maximum principle holds. We refer the reader to Thm. 1.1 in \cite{fontenele2001tangency} for details. 
\end{remark}

\begin{definition}
Let $\Sigma$ be an entire translating graphs that satisfies Equation \eqref{f-flow}. Then, the end of $\Sigma$ is \emph{smoothly asymptotic}
to the ``bowl''-type soliton if $\Sigma$ can be expressed outside a ball as a vertical graph of a function $u_{\Sigma}$ such that
\begin{align}\label{smoothly asym. to the bowl}
    u_{\Sigma}(x)=\dfrac{|x|^2}{2}-\frac{c}{2}\ln(|x|^2)+O\left(|x|^{-1}\right),\mbox{ as }|x|\to\infty.
\end{align}
Note that the normalization $f(0,1)=1$ is used and $c 
\coloneqq \dfrac{\partial f}{\partial \lambda_1}(0,1)$.
\end{definition}

\begin{remark}\label{rem rot. inv}
As is noted in \cite{jtsmaximumprinciple} (see sec. 4), Equation \eqref{f-trans} is invariant under rotation fields that fix the $x_{n+1}$-axis in $\R^{n+1}$. Consequently, under the hypothesis of Theorem \ref{T3}, it is enough to show that $\Sigma$ is symmetric along the plane $\set{x_1=0}$ to obtain that $\Sigma$ is rotationally symmetric. 
\end{remark}

We will now use Alexandrov's moving plane method to establish uniqueness. To this end, we recall the following definitions used in \cite{MHS} and \cite{martinez2022equilibrium}:

\begin{itemize}
    \item  $\Pi_t \coloneqq \set{x\in\R^{n+1}:\textbf{p}(x)=t}$, where $\textbf{p}(x_1,\ldots,x_{n+1})=x_1$ is the projection onto the first coordinate. In addition, $\Pi \coloneqq \Pi_0$.
    \item $Z_t \coloneqq \set{x_{n+1}>t}$.
    \item Let $A \subset \R^{n+1}$ be an arbitrary subset. Then we set $A_+(t) \coloneqq \set{x\in A: \textbf{p}(x)\geq t}$, $A_{-}(t) \coloneqq \set{x\in A: \textbf{p}(x)\leq t}$ and $\delta_t(A):=A\cap\Pi_t$. Note that $A_+(t)$ and $A_-(t)$ are the right hand side and the left hand side, respectively, along $\Pi_t$ of $A$ (see Definition \ref{Def Order} given below).   
	\item The $1$-parameter families of right (respectively, left) reflections of $A$, respectively, along the hyperplane $\Pi_t$  are given by $A_{+}^{*}(t):=\set{(2t-x_1,x_2,\ldots,x_{n+1})\in\R^{n+1}:(x_1,\ldots,x_{n+1})\in A_+(t)}$ (respectively, $A_{-}^{*}(t) := \set{(2t-x_1,x_2,\ldots,x_{n+1})\in\R^{n+1}:(x_1,\ldots,x_{n+1})\in A_{-}(t)}$).
	\item We denote by $\pi:\R^{n+1}\to \Pi_0 \subset\R^{n+1}$ the orthogonal projection given by $\pi(x_1,\ldots,x_{n+1})=(0,x_2,\ldots,x_{n+1})$.
\end{itemize}

\begin{definition}\label{Def Order}
Let $A,B$ be two subsets of $\R^{n+1}$. Then, it will be said  ``$A$ is on the right side of $B$'’ (denoted by $B\leq A$) if, and only if, for every $x\in \Pi=\set{x_1=0}$ such that
	\begin{align*}
		\pi^{-1}(\set{x})\cap A\neq \emptyset \mbox{ and }	\pi^{-1}(\set{x})\cap B \neq\emptyset,
	\end{align*}
	we have that 
	\begin{align}\label{def order}	
		\sup\set{\mathbf{p}(p): p\in \pi^{-1}(\set{x})\cap B}\leq \inf\set{\mathbf{p}(p): p\in\pi^{-1}(\set{x})\cap A}.
	\end{align}
	
\end{definition}

\begin{remark}
	The proof of Theorem \ref{T3} uses the method of moving planes of Alexandrov in the spirit of \cite{schoen1983uniqueness}, \cite{MHS} and \cite{martinez2022equilibrium}. This method specifically requires three properties to be applied: 
 \begin{itemize}
     \item  A family of hyperplanes as translating solutions of Equation \eqref{f-trans}. To accomplish this, the Property \ref{e} on $f$ is needed.
     \item Tangential principles in the interior and at the boundary, since we will need to understand how the translators intersect tangentially with their reflections along $\Pi_t$.
     \item ``Enough space'' to start the method. This means that we can find $t>0$ large enough such that the reflection of $\Sigma$ along $\Pi_t$ is a graph over $\Pi$ and $\Sigma_-(t)\leq\Sigma^*_+(t)$. To compare the horizontal distance between $\Sigma_+^*$ and $\Sigma_-$, we will need the asymptotic behavior of the ``bowl''-type solution at infinity.
 \end{itemize}
\end{remark}

Next, we will follow the arguments used in \cite{martinez2022equilibrium} theorem 6, for $f$-translators which are smoothly asymptotic to the ``bowl''-type solution. 

\begin{lemma}\label{lem1}
    There exists $r_0>R$ such that $\Sigma_+(t)$ is a graph over $\Pi$ for every $t>r_0$. 
\end{lemma}
\begin{proof}
    Firstly, we note that for every $t>R$, $\Sigma_+(t)$ possess only one unbounded connected component. If this were false, one could choose a compact component $\Sigma'\subset\Sigma_+(t)$ and a $t$ large enough such that $\Sigma\cap\Pi_t=\emptyset$. Then, by translating $\Pi_t$ until it touches $\Sigma'$ at a first order contact point\footnote{First, this is possible because the first-order contact point coincides with the point at which the two hypersurfaces first touch. In addition, the first order contact point means that the graph and gradient coincide. In particular, the tangent planes coincide because the unit normal vectors of the two hypersurfaces coincide at this point.}, say $\Sigma'\cap\Pi_{t'}=\set{p}$ for some $0<t'<t$, Theorem \ref{T2}, applied to $\Sigma'$ with $\Pi_{t'}$, implies that $\Sigma'$ is totally geodesic contradicting that $\Sigma'$ is strictly convex.
\newline

    Next, by Equation \eqref{smoothly asym. to the bowl}, we have 
    \begin{align*}
    d_xu_{\Sigma}(e_1)\geq\left(1-\dfrac{C}{|x|^2}\right)\pI{x,e_1},
    \end{align*}
    for some constant $C>0$ and $|x|\geq R$. Therefore, by choosing $r_0>R$ large enough such that $1-\dfrac{C}{|x|^2}\geq \varepsilon>0$, it follows that $d_xu_{\Sigma}(e_1)>0$  whenever $\pI{x,e_1}\geq r_0$. Finally, Lemma \ref{lem1} follows since $\Sigma$ is properly embedded and $\Sigma_+(r_1)\cup\pi\left(\Sigma_+(r_1)\right)$ bounds a domain in $\R^{n+1}$.
\end{proof}

From Lemma \ref{lem1}, it follows that $\Sigma^*_+(t)\cap Z_R$ is a vertical graph of a function satisfying 
\begin{align*}
u^*_\Sigma(x)=u_{\Sigma}(2t-x_1,x_2,\ldots,x_n) 
\end{align*}
for any $t>r_0$ .

\begin{lemma}\label{lem2}
Let $a>0$. There exists $r_1>r_0$ such that for $|x|\geq r_1$ and $t>a+x_1$ we have
\begin{align*}
    u^*_{\Sigma}(x)-u_\Sigma(x)\geq \epsilon,
\end{align*}
for some $\epsilon>0$. 
\end{lemma}
\begin{proof}
The proof follows the same method as in \cite{MHS} Step 3 Theorem A. 
%Firstly, it is noted that
%\begin{align*}
%  u^*_{\Sigma}(x)-u_\Sigma(x)=& \dfrac{(2t-x_1)^2-x_1^2}{2}-\dfrac{c}{2}\log\left(\dfrac{(2t-x_1)^2+x_2^2+\ldots+x_n^2}{|x|^2}\right)
%   \\
%   &+O\left(\dfrac{1}{\sqrt{(2t-x_1)^2+x_2^2+\ldots+x_n^2}}\right)-O\left(\dfrac{1}{|x|}\right)
%   \\
%   \geq&\: 2t(t-x_1)-\dfrac{c}{2}\log\left(1+\dfrac{4t(t-x_1)}{|x|^2}\right)
%-C\left(\dfrac{1}{\sqrt{(2t-x_1)^2+x_2^2+\ldots+x_n^2}}+\dfrac{1}{|x|}\right),
%\end{align*}
%for some constant $C>0$. Recall that the above equations holds for $|x|>r_0$ and $x_1\leq t$.
%
%On the other hand, the terms in the last inequality can be estimated as follows 
%\begin{align*}
%\dfrac{c}{2}\log\left(1+\dfrac{4t(t-x_1)}{|x|^2}\right)\leq \dfrac{2ct(t-x_1)}{|x|^2},
%\end{align*}
%and $(2t-x_1)^2+x_2^2+\ldots x_n^2\geq t^2+(t-x_1)^2+x_2^2+\ldots+x_n^2\geq r^2$. Consequently, it turns out that
%\begin{align*}
%   u^*_{\Sigma}(x)-u_\Sigma(x)\geq 2t(t-x_1)-\dfrac{2ct(t-x_1)}{|x|^2}-\dfrac{2C}{|x|}.  
%\end{align*}
%
%Finally, by hypothesis it holds $t-x_1>a$ and let $\epsilon>0$. Then, by choosing $r_1>r_0$ such that 
%\begin{align*}
%    2t(t-x_1)-\dfrac{2ct(t-x_1)}{|x|^2}-\dfrac{2C}{|x|}\geq \dfrac{2ta(r_1^2-c)-2Cr_1}{r_1^2}\geq\epsilon,
%\end{align*}
%the result holds.
\end{proof}
From Lemma \ref{lem2} we have that for every $a>0$ and $t\geq r_1$ 
\begin{align*}
\Sigma_-(t+a)\cap\set{x_1\leq t}\leq \Sigma_+^*(t+a)\cap\set{x_1\leq t}. 
\end{align*}
Moreover, since $\Sigma_+(t)$ is a graph over $\Pi$ for $t\geq r_1$, it follows that 
\begin{align*}
\Sigma_-(t+a)\cap\set{t\leq x_1\leq t+a}\leq \Sigma_+^*(t+a)\cap\set{t\leq x_1\leq t+a}. 
\end{align*}
In conclusion, by setting
\begin{align*}
    \mathcal{A} \coloneqq \set{t\in [0,\infty): \Sigma_+(t) \mbox{ is a graph over }\Pi\mbox{ and }\Sigma_-(t)\leq \Sigma_+^*(t)},
\end{align*}
we see that $\mathcal{A}\neq\emptyset$ since $t+a\in\mathcal{A}$ for $t\geq r_1$. 

\begin{proof}[Proof of Theorem \ref{T3}]
Firstly, by the Remark \ref{rem rot. inv} it is only necessary to prove that $\Sigma$ is symmetrical about the hyperplane $\Pi=\set{x_1=0}$. Moreover, by the lemmas \ref{lem1} and \ref{lem2}, the set 
\begin{align*}
    \mathcal{A}=\set{t\in [0,\infty): \Sigma_+(t) \mbox{ is a graph over }\Pi\mbox{ and }\Sigma_-(t)\leq \Sigma_+^*(t)},
\end{align*}
is not empty. Then, we apply the same arguments as in \cite{MHS} to obtain $\mathcal{A}$ is closed and open subset of $[0,\infty)$. Consequently, $\mathcal{A}=[0,\infty)$, which means $0\in \mathcal{A}$. In fact, we have obtained $\Sigma_-(0)\leq \Sigma_+^*(0)$, and by analogous arguments, it can be shown that $ \Sigma_{-}^*(0)\leq \Sigma_+(0)$. Note that the combination of these two properties implies that $\Sigma$ is symmetric with respect the hyperplane $\Pi$. 
\end{proof}

\section{A Degenerate Example: $\sqrt[n]{S_n}$.}\label{Sn}
As seen in the previous section, the asymptotic behavior of the ``bowl''-type solution is that of a paraboloid when the velocity function is nondegenerate. 
\newline

In this section we will show that the situation can be very different for nondegenerate speeds, taking as an example the function $\sqrt[n]{S_n}$. 
\begin{theorem}\label{Asymp S_n slope}
The slope of the  ``bowl''-type $\sqrt[n]{S_n}$-translator is given by 
\begin{align*}
   v(r)=\begin{cases}
        e^\frac{r^2}{2}+O\left(\sqrt{e^\frac{r^2}{2}-1}\right),&\mbox{ for }n=2,
        \\
        \dfrac{r^3}{3}+O(1),&\mbox{ for }n=3,
        \\
        \left(\frac{n-2}{n}\right)^\frac{1}{n-2}r^{\frac{n}{n-2}}+O\left(\dfrac{1}{r^\frac{n}{n-2}}\right),&\mbox{ for }n\geq4. 
    \end{cases},\mbox{ as }r\to\infty
\end{align*}
\end{theorem}
The proof of this theorem will be split into the following propositions. First, recall that the slope of the ``bowl''-type $\sqrt[n]{S_n}$-translator
satisfies the ODE
\begin{align}\label{S_n-rans}
	\begin{cases}
		v'=(1+v^2)\left(\dfrac{r}{v}\right)^{n-1},
		\\
		v(0)=0.
	\end{cases}
\end{align} 	

We start with the case of $n=2$. 
\begin{proposition}\label{Prop S_2}
 For $n=2$ the smooth solution to \eqref{S_n-rans} satisfies   \begin{align*}
    v(r)=e^\frac{r^2}{2}+O\left(\sqrt{e^\frac{r^2}{2}-1}\right),\mbox{ as }r\to\infty. 
 \end{align*}
\end{proposition}
\begin{proof}
Firstly, for $n=2$, we note that Equation \eqref{S_n-rans} has the form
\begin{align}\label{S_2 ODE}
	v'=(1+v^2)\dfrac{r}{v}. 
\end{align} 
Therefore, $v=\sqrt{e^\frac{r^2}{2}-1}$ for $r>0$. Then, Proposition \ref{Prop S_2} holds by the explicit expression of the solution.
\end{proof}

\begin{proposition}\label{Prop S_3}
For $n=3$ the smooth solution to \eqref{S_n-rans} satisfies   \begin{align*}
    v(r)=\dfrac{r^3}{3}+O(1),\mbox{ as }r\to\infty. 
 \end{align*}  
\end{proposition}
\begin{proof}
For $n\geq 3$, Equation \eqref{S_n-rans} can be solved implicitly by
\begin{align}
	\dfrac{v^{n-2}}{n-2}-\int\limits_0^{v}\dfrac{t^{n-3}}{1+t^2}dt=\dfrac{r^n}{n}. 
\end{align}
Particularly, for $n=3$, we have the implicit equation  
\begin{align*}
	v-\arctan(v)=\dfrac{r^3}{3}. 
\end{align*}
On the other hand, the function $f(x)=x-\arctan(x)$ has the following asymptotic expansion
\begin{align*}
	f(x)=x-\dfrac{\pi}{2}+O(x^{-1}), \mbox{ as }x\to\infty.
\end{align*}
Then, an easy calculation reveals that $f^{-1}(x)$ has the following asymptotic given by
\begin{align*}
	f^{-1}(x)=x+\dfrac{\pi}{2}-O(x^{-1}), \mbox{ as }x\to\infty. 
\end{align*} 
Consequently, we obtain that the asymptotic expansion of $v$ is given by
\begin{align*}
	v(r)=\dfrac{r^3}{3}+\dfrac{\pi}{2}-O(r^{-3}), \mbox{ as }r\to\infty. 
\end{align*} 
\end{proof}

\begin{proposition}\label{T S_4}
	For $n>3$, the asymptotic behavior of the solution to \eqref{S_n-rans} is given by
	\begin{align*}
			v(r)=\left(\frac{n-2}{n}\right)^\frac{1}{n-2}r^{\frac{n}{n-2}}+O\left(\dfrac{1}{r^\frac{n}{n-2}}\right),\mbox{ as }r\to\infty.
		\end{align*}
\end{proposition}
\begin{proof}
The proof is divided in several steps.
	\begin{step}
	 $v(r)\geq \left(\dfrac{n-2}{n}\right)^{\frac{1}{n-2}}r^\frac{n}{n-2}$ for $r> 0 $. 
	\end{step}	
Firstly, we are going to show that the function $w(r)=\left(\dfrac{n-2}{n}\right)^{\frac{1}{n-2}}r^{\frac{n}{n-2}}$ is a sub-solution to Eq. \eqref{S_n-rans} for $r\geq R(n):=\left(\dfrac{n-2}{2}\right)^{\frac{n-2}{2}} $. Indeed, we note that 
		\begin{align*}
			&w'-(1+w^2)\left(\dfrac{r}{w}\right)^{n-1}
			\\
			&=\left(\dfrac{n}{n-2}\right)^\frac{n-3}{n-2}r^\frac{2}{n-1}-\left(1+\left(\dfrac{n-2}{n}\right)^\frac{2}{n-2}r^\frac{2n}{n-2}\right)\left(\frac{n-2}{n}\right)^{-\frac{(n-1)}{n-2}} r^\frac{-2(n-1)}{n-2}
			\\
			&=-\dfrac{1}{\left(\dfrac{n-2}{n}\right)^\frac{n-1}{n-2}r^\frac{2(n-1)}{n-2}}<0.
		\end{align*}
 Since $w(0)=v(0)=0$, we get $v(r)\geq w(r)$ for all $r\geq R(n)$.
		
\begin{step}
For $n>3$, 	$v(r)=\left(\dfrac{n-2}{n}\right)^{\frac{1}{n-2}}r^\frac{n}{n-2}+o\left(r^\frac{n}{n-2}\right)$ as $r\to\infty$.
\end{step}				
We argue by contradiction which means there exists $\varepsilon>0$ such that for all $r_0\gg 1$ it holds 
		\begin{align*}
			(1+\varepsilon)w(r)\leq v(r),\mbox{ for all }r\geq r_0.  
		\end{align*} 
In particular,
		\begin{align*}
			v'=(1+v^2)\left(\dfrac{r}{v}\right)^{n-1}\leq \dfrac{r^{n-1}w^{-(n-3)}}{(1+\varepsilon)^{n-1}}=\dfrac{1}{(1+\varepsilon)^{n-1}}\left(\dfrac{n-2}{n}\right)^{-\frac{n-3}{n-2}}r^\frac{2}{n-2}.
		\end{align*}
Then, by the previous step, we have 
\begin{align*}
\left(\dfrac{n-2}{n}\right)^\frac{1}{n-2}	r^\frac{n}{n-2}\leq v\leq \dfrac{\left(\frac{n-2}{n}\right)^\frac{1}{n-2}}{(1+\varepsilon)^{n-1}} r^\frac{n}{n-2}.
\end{align*} 	
Note that this is impossible  since $\varepsilon>0$. Consequently, $v(r)=w(r)+o\left(r^{\frac{n}{n-2}}\right)$ for $r\to\infty$. 
\\
	
Next, since $v(r)=\left(\dfrac{n-2}{n}\right)^\frac{1}{n-2}r^\frac{n}{n-2}+o(r^\frac{n}{n-2})$ as $r\to\infty$, we can find a non-negative function $\varphi(r)$ with the following property: for every $C>0$ and all $r\gg 1$ it holds $|\varphi(r)|\leq Cr^\frac{n}{n-2}$ and
	\begin{align*}
		v(r)=\left(\dfrac{n-2}{n}\right)^\frac{1}{n-2}r^\frac{n}{n-2}+\varphi(r), \mbox{ for }r\geq r_0.
	\end{align*}
In addition, $\varphi$ satisfies the following equation
	\begin{align*}
		\varphi'&=v'-\left(\dfrac{n}{n-2}\right)^\frac{n-3}{n-2}r^\frac{2}{n-2}
		%\\
		%&=(1+v^2)\left(\dfrac{r}{v}\right)^{n-1}-\left(\dfrac{n}{n-2}\right)^\frac{n-3}{n-2}r^\frac{2}{n-2}
		\\
		&=\left(1+\left(\left(\frac{n-2}{n}\right)^\frac{1}{n-2}r^\frac{n}{n-2}+\varphi\right)^2\right)\dfrac{r^{n-1}}{\left(\left(\dfrac{n-2}{n}\right)^\frac{1}{n-2}r^\frac{n}{n-2}+\varphi\right)^{n-1}}-\left(\frac{n}{n-2}\right)^\frac{n-3}{n-2}r^\frac{2}{n-2}.
	\end{align*}

\begin{step}
	$\varphi\to 0$ as $r\to\infty$. 
\end{step}
We argue by contradictions, this means there exist $\varepsilon>0$ such that for all $r_0\gg 1$ and every $C>0$ it holds 
\begin{align*}
	\varepsilon\leq \varphi(r)\leq Cr^\frac{n}{n-2} ,\mbox{ for all }r\geq r_0.  
\end{align*} 
In particular, by letting $C\to 0$ in 
\begin{align*}
	\varphi'=&\left(1+\left(\left(\frac{n-2}{n}\right)^\frac{1}{n-2}r^\frac{n}{n-2}+\varphi\right)^2\right)\dfrac{r^{n-1}}{\left(\left(\dfrac{n-2}{n}\right)^\frac{1}{n-2}r^\frac{n}{n-2}+\varphi\right)^{n-1}}-\left(\frac{n}{n-2}\right)^\frac{n-3}{n-2}r^\frac{2}{n-2}
\\
\leq&\left(1+r^\frac{2n}{n-2}\left(\left(\frac{n-2}{n}\right)^\frac{1}{n-2}+C\right)^2\right)\dfrac{r^{n-1}}{\left(\left(\dfrac{n-2}{n}\right)^\frac{1}{n-2}r^\frac{n}{n-2}+\epsilon\right)^{n-1}}-\left(\frac{n}{n-2}\right)^\frac{n-3}{n-2}r^\frac{2}{n-2}
\\
=&\left( \left(\dfrac{n}{n-2}\right)^\frac{n-1}{n-2}\left( \left(\dfrac{n-2}{n}\right)^\frac{1}{n-2}+C \right)^2-\left(\dfrac{n}{n-2}\right)^\frac{n-3}{n-2}\right)r^\frac{2}{n-2}
\\
&-\dfrac{(n-1)\left(\dfrac{n}{n-2}\right)^\frac{n}{n-2}\left( \left(\dfrac{n-2}{n}\right)^\frac{1}{n-2}+C \right)^2\epsilon}{r}+O\left(\dfrac{1}{r^\frac{(n-1)}{(n-2)}}\right),
\end{align*}	
we observe that
\begin{align*}
	\varphi'\leq -(n-1)\left(\dfrac{n}{n-2}\right)\dfrac{\varepsilon}{r}+O\left(\dfrac{1}{r^\frac{(n-1)}{(n-2)}}\right)\leq 0.
\end{align*}
Therefore, $\varphi$ is strictly decreasing but this contradicts that $\varepsilon\leq \varphi$. 

\begin{step}
	$\varphi\geq \left(\dfrac{n}{n-2}\right)^\frac{1}{n-2}\dfrac{1}{r^\frac{n}{n-2}}$ for large enough $r$. 
\end{step}	
 Let $r_1\gg 1$ and $C>0$ such that  
		\begin{align*}
			0\leq \varphi(r)\leq Cr^\frac{n}{n-2}, \mbox{ for all }r\geq r_1. 
		\end{align*}
		Then, we may estimate $\varphi$ by letting $C\to 0$ in the following expression
		\begin{align*}
			\varphi'=&\left(1+\left(\left(\frac{n-2}{n}\right)^\frac{1}{n-2}r^\frac{n}{n-2}+\varphi\right)^2\right)\dfrac{r^{n-1}}{\left(\left(\dfrac{n-2}{n}\right)^\frac{1}{n-2}r^\frac{n}{n-2}+\varphi\right)^{n-1}}-\left(\frac{n}{n-2}\right)^\frac{n-3}{n-2}r^\frac{2}{n-2}
			\\ 
			\geq&\dfrac{1+r^\frac{2n}{n-2}\left(\frac{n-2}{n}\right)^\frac{2}{n-2}}{r^\frac{2(n-1)}{n-2}\left(\left(\frac{n-2}{n}\right)^\frac{1}{n-2}+C\right)^{n-1}}-\left(\frac{n}{n-2}\right)^\frac{n-3}{n-2}r^\frac{2}{n-2}.
			\\
			=&\left(\dfrac{\left(\frac{n-2}{n}\right)^\frac{2}{n-2}}{\left(\left(\frac{n-2}{n}\right)^\frac{1}{n-2}+C\right)^{n-1}}  -\left(\dfrac{n}{n-2}\right)^\frac{n-3}{n-2}\right)r^\frac{2}{n-2}+\dfrac{1}{\left(\left(\frac{n-2}{n}\right)^\frac{1}{n-2}+C\right)^{n-1}r^\frac{2(n-1)}{n-2}}.
		\end{align*}
	Therefore, we have that 
	\begin{align*}
		\varphi'\geq \left(\dfrac{n}{n-2}\right)^\frac{n-1}{n-2}\dfrac{1}{r^\frac{2(n-1)}{n-2}},\mbox{ as }r\to\infty.
	\end{align*}
Then, by integrating form $r$ to $\infty$ we finally obtain 
\begin{align*}
	\varphi\geq \left(\dfrac{n}{n-2}\right)^\frac{1}{n-2}\dfrac{1}{r^\frac{n}{n-2}}.
\end{align*}

\begin{step}
	$\varphi=O\left(\dfrac{1}{r^\frac{n}{n-2}}\right)$.
\end{step}

We note that for $r_0\gg 1$ it holds 
\begin{align*}
	\left(\dfrac{n}{n-2}\right)^\frac{1}{n-2}\dfrac{1}{r^\frac{n}{n-2}}\leq \varphi(r)\leq Cr^\frac{n}{n-2},\mbox{ for all }r\geq r_0.  
\end{align*} 
Next,  by letting $C\to 0$ in 
\begin{align*}
\varphi'=&\left(1+\left(\left(\frac{n-2}{n}\right)^\frac{1}{n-2}r^\frac{n}{n-2}+\varphi\right)^2\right)\dfrac{r^{n-1}}{\left(\left(\dfrac{n-2}{n}\right)^\frac{1}{n-2}r^\frac{n}{n-2}+\varphi\right)^{n-1}}-\left(\frac{n}{n-2}\right)^\frac{n-3}{n-2}r^\frac{2}{n-2}
\\
\leq&\left(1+r^\frac{2n}{n-2}\left(\left(\frac{n-2}{n}\right)^\frac{1}{n-2}+C\right)^2\right)\dfrac{r^\frac{2(n-1)^2}{n-2}}{\left(\left(\dfrac{n-2}{n}\right)^\frac{1}{n-2}r^\frac{2n}{n-2}+\left(\dfrac{n}{n-2}\right)^\frac{1}{n-2}\right)^{n-1}}
\\
&-\left(\frac{n}{n-2}\right)^\frac{n-3}{n-2}r^\frac{2}{n-2},
\end{align*}
we deduce that
\begin{align*}
	 \varphi'\leq -\dfrac{(n-2)\left(\dfrac{n}{n-2}\right)^\frac{n-1}{n-2}}{r^\frac{2(n-1)}{n-2}}+O\left(\dfrac{1}{r^\frac{4(n-1)}{n-2}}\right).
\end{align*} 	
Finally, by integrating over $r$ to $\infty$, we obtain
\begin{align*}
\varphi\leq \dfrac{ (n-2)\left(\dfrac{n}{n-2}\right)^\frac{1}{n-2}}{r^\frac{n}{n-2}} .	
\end{align*}
\end{proof}
This concludes the proof of Proposition \ref{S_n-rans}.

\begin{remark}
    We note that the proof of the uniqueness theorem presented in section \ref{sec: Uniqueness} also applies for entire strictly convex solutions smoothly asymptotic to the ``bowl''-type translators of the $\sqrt[n]{S_n}$-flow. 
    To see this, we choose $R>0$. Then, the entire $\sqrt[n]{S_n}$-translator can be written as a vertical graph $u_\Sigma:\R^n\setminus B_R(0)\to \R$ such that
    \begin{align*}
        u_\Sigma=h_1(r^2)+O\left(h_2(r^2)\right), \mbox{ as }r\to\infty,
    \end{align*}
    where $|x|=r$ and
     \begin{align*}
         h_1(r^2)=\begin{cases}
\int e^\frac{r^2}{2}dr,& \mbox{ for }n=2,
\\
\frac{1}{12}r^4,& \mbox{ for }n=3,
\\
\dfrac{(n-2)^\frac{n-1}{n-2}}{2(n-1)n^\frac{1}{n-2}}r^\frac{2(n-1)}{n-2},& \mbox{ for }n\geq 4.
         \end{cases},\:
        h_2(r^2)=\begin{cases}
          \int\sqrt{e^\frac{r^2}{2}-1}dr,&\mbox{ for }n=2,
          \\
          r,&\mbox{ for }n=3,
          \\
          \frac{1}{r^\frac{2}{n-2}},&\mbox{ for }n\geq 4. 
        \end{cases}
     \end{align*}
Then, Lemma \ref{lem1} holds for $f=\sqrt[n]{S_n}$, since 
     \begin{align*}
         d_xu_\Sigma(e_1)\geq \left(h_1'(|x|^2)-Ch_2(|x|^2)\right)\pI{x,e_1},
     \end{align*}
    and 
    \begin{align}\label{h_1'-h_2'}
    0<h_1'(|x|^2)-Ch_2'(|x|^2)=\begin{cases}
            e^\frac{|x|^2}{2}-C\sqrt{e^\frac{|x|^2}{2}-1},& \mbox{ for}n=2
        \\
        \dfrac{|x|^3}{3}-C,&  \mbox{ for }n=3
        \\
        \dfrac{(n-1)}{(n-2)^\frac{n-3}{n-2}}|x|^{\frac{2}{n-2}}-\dfrac{C}{|x|^\frac{n}{n-2}},&\mbox{ for }n\geq 4.   
    \end{cases}, 
    \end{align}
 for large enough $r\geq R$ and some positive constants $C$. 
\\

In addition, we have that Lemma \ref{lem2} also holds, since by the mean value theorem, it follows that 
    \begin{align*}
   u^*_{\Sigma}(x)-u_\Sigma(x)&=2(t-x_1)\dfrac{\partial u_{\Sigma}}{\partial x_1}(\xi,x_2,\ldots,x_n),
\end{align*}
for some $\xi\in (x_1,2t-x_1)$ with $|x|>r_0$. 
\newline

Then, by writing $(\xi,x')=(\xi,x_2,\ldots,x_n)$,  the asymptotic expression of $u_{\Sigma}$ gives us
\begin{align*}
2(t-x_1)\dfrac{\partial u_{\Sigma}}{\partial x_1}(\xi,x_2,\ldots,x_n)&\geq 4(t-x_1)\xi\left(h_1'(|(\xi,x')|^2)-Ch_2'(|(\xi,x')|^2)\right)
\\
&\geq 4ar_0(h_1(|(\xi,x')|^2)'-h_2'(|(\xi,x')|^2)). 
\end{align*}
Finally, by choosing $r_1\geq r_0$ as in \eqref{h_1'-h_2'}, the difference $h_1(|(\xi,x')|^2)'-h_2'(|(\xi,x')|^2)$ is uniformly bounded from below. Consequently, the proof of Theorem \ref{T3} holds.     
\end{remark}

\section{Wing-like translators}\label{wing-like sec}\,
An exotic translating solution to the mean curvature is the \emph{wing-like} translator discussed in \cite{CSS}. This solution is connected but not graphical, and is rather the union of two graphs each defined on the complement of a ball $\R^n-B_R$. The upper and lower halves are each asymptotic to paraboloids. 
\\

In this section, we will discuss the existence and the asymptotic behavior of \emph{wing-like} translators for \emph{nondegenerate} speeds (see definition \ref{Nondegenerate}). In particular, we will see that the upper and lower branches need not always have the same asymptotics.
\newline

More precisely, a \emph{wing-like} solution is a non-convex rotationally symmetric hypersurface in $\R^{n+1}$ with compact boundary (possible empty) along the $x_{n+1}$-axis which is at distance $R>0$ from the origin that satisfies Eq. \eqref{f-trans}. 
\\
Furthermore, by removing a large enough ball from a \emph{wing-like} solution we have at possible two branches given by vertical graphs:
\begin{itemize}
    \item We are going to show that the upper branch is always asymptotic to the ``bowl’'-type solution.
    \item Meanwhile, the lower branch has shown different behaviors at infinity with respect to the speed function $f$. 
    
    In particular, we will show that for the function $f=\sqrt[k]{S_k}$ with $k$ even the lower branch is asymptotic to symmetric reflection of ``bowl'' type solution, for $k$ odd, the lower posses a boundary component, and for the function $f=Q_{k+1,k}$, the lower branch whose gradient vanishes at infinity. 
\end{itemize}

\begin{remark}
We note that the ODE satisfying a rotationally symmetric $Q_{k+1,k}$-translator has as solution $v=0$, but recall that the horizontal hyperplanes are not vertical translators and that $Q_{k+1,k}$ is not well defined at $0$.
\end{remark}

\begin{remark}
    It is an open problem to characterize the dichotomy of the lower branch of wing-like solutions for any \emph{nondegenerate} speed $f$. 
\end{remark}

Now we will follow the construction given in \cite{CSS}. Firstly, we will start with the a small non-convex portion of the \emph{wing-like} solution. 

% \textcolor{green}{We willl outline how we construct the winglike solution in several steps:
% \begin{enumerate}
%     \item Construction of a small portion of the ``wing-like" solution near its tip by viewing it as a graph over the $x_{n+1}$-axis
%     \item 
% \end{enumerate}
% }

\begin{step}\label{small piece}
Construction of a small portion of the \emph{wing-like} $f$-translator as a graph over the $x_{n+1}$-axis.
\end{step}
% \begin{proof}
   At the point where the tangent space is not orthogonal to $e_{n+1}$, the translator can be represented locally as a graph of a function $r:(a,b)\to(0,\infty)$ over the $x_{n+1}$-axis by 
   \begin{align*}
       \bigcup_{x_{n+1}}r(x_{n+1})\cdot\mathbb{S}^{n-1}\times\{x_{n+1}\}. 
   \end{align*}
 Note that $r(x_{n+1})$ represents the radius of the hypersurface given its last coordinate. 

Next, we are going to find the ODE the $r$ satisfies. Firstly, at points where the tangent space is not parallel to $e_{n+1}$, the \emph{wing-like} $f$-translator can be described as a graph of a rotationally symmetric function $u(r)$ over the hyperplane $\set{x_{n+1}=0}$. 
\newline

Note that by construction $u\circ r(x_{n+1})=x_{n+1}$, and by the chain rule, we have  the following equations:
\begin{align*}
   u'(r)=\dfrac{1}{r'}\mbox{ and }u''(r)=-\dfrac{r''}{(r')^3}. 
\end{align*}
Consequently, since a \emph{wing-like} solution is not convex, Eq. \eqref{f-trans} has the form \eqref{wing-like f}. Then, since $u$ satisfies equation \eqref{f-trans ODE g}, we obtain
\begin{align*}
    r''=-(1+(r')^2)g\left(r^{-1},r'\right),
\end{align*}
where $g\left(r^{-1},r'\right)=r'g\left((rr')^{-1},1\right)$ since $g(y,z)$ is $1$-homogeneous. We note that in previous sections we denote $g(y)$ instead of $g(y,1)$ to economize notation.   
\newline

In addition, at points of the \emph{wing-like} $f$-translator where the tangent space is vertical, we may argue in the same way as before by considering the branches separately.
\newline

Finally, we claim that there exists $\varepsilon>0$ and a strictly convex solution to the problem 
\begin{align}\label{r ODE}
\begin{cases}
    r''=-(1+(r')^2)g\left(r^{-1},r'\right),\: x_{n+1}\in (h_0-\varepsilon,h_0+\varepsilon),
    \\
    r(h_0)=R,\:r'(h_0)=0. 
\end{cases},
\end{align}
where $h_0\in\mathbb{R}$ and $R>0$. Note that a different choice of $h_0$ corresponds to a vertical translation of the \emph{wing-like} $f$-translator. Therefore, we will assume that $h_0=0$. 
\newline

Indeed, since Eq. \eqref{r ODE} is not degenerate and the he right-hand side is at least of class $\mathcal{C}^{2}$, the classical theory of ODEs holds. This means that there exist $\varepsilon>0$ and a solution to Eq. \eqref{r ODE} in $(-\varepsilon,\varepsilon)$. 
\newline

Then, by shrinking $\varepsilon>0$ if necessary, we may assume that this solution is strictly convex. This is due to continuity and the fact that
\begin{align*}
    r''(0)=-g\left(R^{-1},0\right)>0.
\end{align*}
To see this, we note that by taking derivatives with respect $z$ in Eq.\eqref{wing-like f}, we have$f_xg_z=1$. Then, since $f$ is strictly increasing, $g_z(y,z)>0$. Therefore, by recalling that $$g(1,1)=0\mbox{ and }g(R^{-1},0)<0$$ the claim holds. Completing the construction of the small portion of the ``wing''-like $f$-translator.   
% \end{proof}
\begin{remark}\label{k even}
In the particular case of $f=\sqrt[k]{S_k}$, the ODE that $r$ satisfies is 
\begin{align*}
    r''=-(1+(r')^2)\left(\dfrac{(r')^kr^{k-1}}{\binom{n-1}{k-1}}-\dfrac{(n-k)}{kr} \right).
\end{align*}
Therefore, when $k$ is an even number, we have that $\tilde{r}=r(-x_{n+1})$ is also a solution to Eq. \eqref{r ODE}. 
\end{remark}

\begin{proposition}
     The principal curvatures of the upper half small portion of the \emph{wing-like} $f$-translator belong into the cone $$\set{f\mbox{ satisy properties }\ref{a}-\ref{e}\mbox{ and }f(0,1)>0}.$$
\end{proposition}

\begin{proof}
Firstly, we note that Eq. \eqref{Principal curvatures} implies that 
\begin{align*}
   \lambda_1=\frac{-r''}{(1+r'^2)^{3/2}}\mbox{ and } \lambda_i=\frac{1}{r\sqrt{1+r'^2}}. 
\end{align*}
are the principal curvatures of the small portion of the \emph{wing-like} $f$-translator.
\\

Then, by differentiating them
\begin{align*}
    \lambda_1'=\frac{-r'''}{(1+r'^2)^{3/2}}+\frac{3(r'')^2r'}{(1+r'^2)^{5/2}}  
\mbox{ and }
    \lambda_i'=\frac{-r'}{r^2\sqrt{1+r'^2}}-\frac{r'r''}{r(1+r'^2)^{3/2}}  
\end{align*}

Next, we will be evaluate the principal curvatres at $u=0$\footnote{ $u>0$ is the upper part of the winglike solution and $u<0$ is the lower.}. Differentiating $r''$, we have
\begin{align*}
r'''=-2r'r''g-(1+r'^2)g_y\frac{r'}{r^2}-(1+r'^2)g_zr'',
\end{align*}
where $g,g_y,g_z$ are all evaluated at the point $\left(r^{-1},r'\right)$. Therefore, we obtain
\begin{align*}
r'''(0)=-g_z\left(R^{-1},0\right)r''(0)=g_z\left(R^{-1},0\right)g\left(R^{-1},0\right).
\end{align*}
Finally, by recalling $r(0)=R$, $r'(0)=0$ and $r''(0)=-g(R^{-1},0)<0$, we see that
\begin{align*}
\lambda_1'(0)=-g_z\left(R^{-1},0\right)g\left(R^{-1},0\right)>0,
\mbox{ and }
\lambda_i'(0)=0.
\end{align*}
Therefore, since in the upper half of the small portion $\set{u>0}$, $r$ is strictly decreasing. We obtain that $\lambda_i'>0$ for all $i=1,\ldots,n$ finalizing the proposition.
\end{proof}

\begin{step}
Construction of the upper branch of the \emph{wing-like} $f$-translator. 
\end{step}
 
% \begin{proof}
We have a solution on some interval $u\in [0,\varepsilon)$. To continue the solution, we may revert to the standard representation 
\begin{align*}
    \begin{cases}
        v'=(1+v^2)g\left(\dfrac{v}{r}\right)\\
        v(r_0)=v_0
    \end{cases}
\end{align*}
where we may choose $r_0\in(R,R+\varepsilon)$ arbitrarily. Then, since $v\to\infty
$ as $r\to R^+$, we may assume $\dfrac{v_0}{r_0} \geq 1$, by choosing $r_0$ sufficiently close to $R.$ By standard ODE theory we have existence and uniqueness as long as the right hand side is well-defined.
\newline

If $\dfrac{v_0}{r_0}>1$, the right hand side of the ODE is initially negative, hence the function decreases until $g\left(\dfrac{v}{r}\right)=0$ (i.e.: $\dfrac{v}{r}=1$). So we only need to discuss the case $\dfrac{v_0}{r_0}=1$. If this is the case, then the function $v_+=r$ is a supersolution to the ODE, and $v_-=\dfrac{r}{f(1,1)}$ is a subsolution, as can be seen in \cite{rengaswami2021rotationally} Section 7.1. The proof that $\dfrac{v}{r}\to 1$ is just a special case of Proposition 2 of \cite{rengaswami2021rotationally}.
% \end{proof}
\\

On the other hand, for the lower branch of the \emph{wing-like} $f$-translator, we have a solution on some interval $(-\varepsilon,0]$. As with the upper branch, we may revert to the standard representation 
\begin{align*}
    \begin{cases}
        v'=(1+v^2)g\left(\dfrac{v}{r}\right),\\
        v(r_0)=v_0.
    \end{cases},
\end{align*}
where $r_0\in(R,R+\varepsilon)$. Then, Since $v\to-\infty$ as $r\to R^+$, we may assume $\dfrac{v_0}{r_0} \leq -1$, by choosing $r_0$ sufficiently close to $R.$ By standard ODE theory we have existence and uniqueness as long as the right hand side is well-defined.
\newline

\begin{step}
  Construction and asymptotic behavior of the lower branch of the \emph{wing-like} $\sqrt[k]{S_k}$-translator for $k$ even.
\end{step}

Firstly, by construction the principal curvatures of the small portion of $W_R$ are given by
\begin{align*}
   \lambda_1=\dfrac{g(r^{-1},r')}{\sqrt{1+(r')^2}}=\dfrac{\dfrac{(r')^kr^{k-1}}{\binom{n-1}{k-1}}-\dfrac{(n-k)}{kr}}{\sqrt{1+(r')^2}}\mbox{ and }\lambda_i=\dfrac{1}{r\sqrt{1+(r')^2}}. 
\end{align*}
Then, since $\lambda\in\Gamma_k$ if, and only if, 
\begin{align*}
    \binom{n-1}{l}\lambda_i^l+\binom{n-1}{l-1}\lambda_1\lambda_i^{l-1}>0,\mbox{ for }l=1,\ldots, k,
\end{align*}
or equivalently, we have
\begin{align*}
    0<\dfrac{\binom{n-1}{l-1}}{r^l(1+(r')^2)^{\frac{l}{2}}}\left(rg(r^{-1},r')+\dfrac{n-l}{l}\right)\Leftrightarrow 0<\dfrac{(r'r)^k}{\binom{n-1}{k-1}}+\dfrac{n(k-l)}{kl}.
\end{align*}
Consequently, since $r'$ only vanish at the origin, it follows that the principal curvatures of the small portion satisfy $\lambda(0)\in\partial\Gamma_k\cap\Gamma_{k-1}$, or equivalently $$\lambda(0)\in\partial\Gamma_k\cap\set{\dfrac{\partial \sqrt[k]{S_k}}{\partial\lambda_i}>0},$$ and $\lambda(p)\in\Gamma_k$ for $p\in W_R\setminus\set{0}$. 
\newline

Next, the ODE for $f=\sqrt[k]{S_k}$ is given by
\begin{align}\label{S_k-trans R>0}
\begin{cases}
 v'=(1+v^2)\left(\dfrac{1}{\binom{n-1}{k-1}}\left(\dfrac{r}{v}\right)^{k-1}-\dfrac{(n-k)v}{kr}\right),\: r\geq r_0,
 \\
 v(r_0)=v_0.
\end{cases}
\end{align}
We note that when $k$ is an even number, then the function $-v$, where $v$ denotes the upper half solution, is also a solution to \eqref{S_k-trans R>0}\footnote{Recall Remark \ref{k even}  and the construction of the \emph{wing-like} solution is by extending the small piece $r$ as the initial data of Eq. \eqref{S_k-trans R>0} in terms of $u$ with $v'=u$.}. 
\newline
In particular, the right hand side of Eq. \eqref{S_k-trans R>0} is always negative and finite. This means the solution exists for all $r\geq r_0$, and the asymptotic expression of $v$ for $k$ even is given by 
\begin{align*}
    v(r)=-\dfrac{r}{\sqrt[k]{\binom{n-1}{k}}}+\dfrac{\binom{n-1}{k}\binom{n-1}{k-1}}{k\sqrt[k]{\binom{n-1}{k}}}\dfrac{1}{r}+O(|x|^{-2}),\mbox{ as }r\to\infty. 
\end{align*}
Consequently, the principal curvatures of $W_R^{-}$ do not belong $\Gamma_k$ for any $k$ as $r\to\infty$.

\begin{figure}
        \centering
        \begin{subfigure}{0.32\textwidth}
                \includegraphics[scale=0.4]{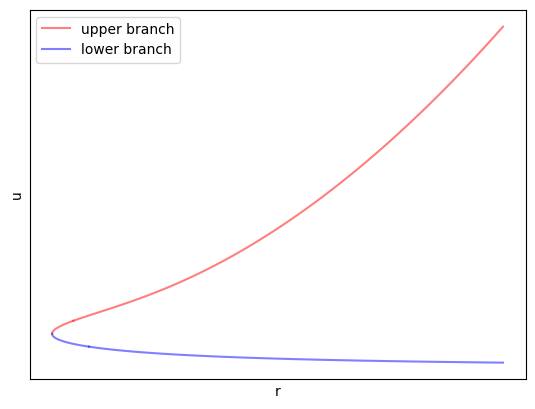}
            \caption{$Q_{k+1,k}$}
        \end{subfigure}
        \begin{subfigure}{0.32\textwidth}
                \includegraphics[scale=0.4]{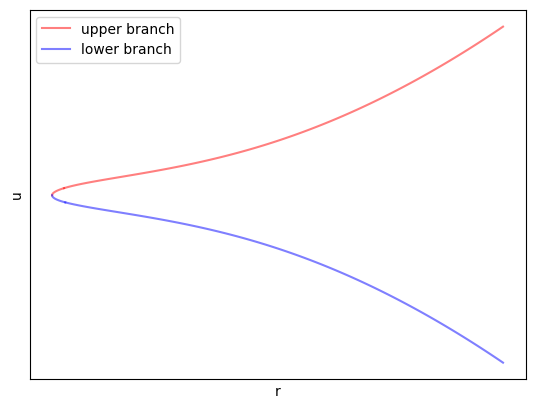}
            \caption{$\sqrt[k]{S_k}$ for $k$ even}
        \end{subfigure}   
        \begin{subfigure}{0.32\textwidth}
                 \includegraphics[scale=0.4]{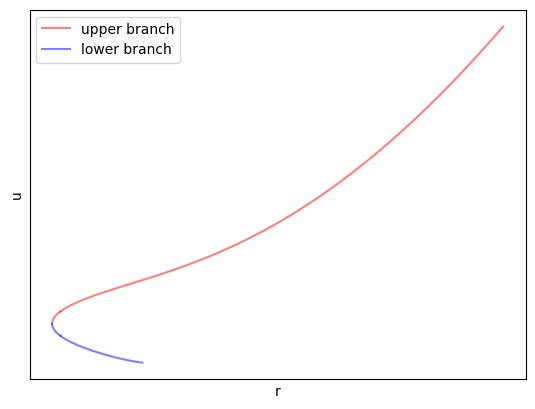}
            \caption{$\sqrt[k]{S_k}$ for $k$ odd}
        \end{subfigure}
        \caption{Profile curves for winglike translators for various speed functions}

\end{figure}

\begin{remark}
It is important to note that when $k$ is odd, the right hand side of Eq. \eqref{S_k-trans R>0} has a singularity when $v=0$. This is because the translator equation  \ref{f-trans ODE f=1} can never be satisfied by any axially symmetric graph that possesses a point where $u'=0$ and $r>0$. Thus, if $u' \to 0$ at some point $r_1>r_0$, $u'$ can never be extended past $r_1$, i.e. such graph will possess a boundary $S^{n-1}$.
\end{remark}

\begin{step}
  Construction and asymptotic behavior of the lower branch of the \emph{wing-like} $Q_{k+1,k}$-translator.
\end{step}
Firstly, the ODE that $r$ satisfies is 
\begin{align*}
    r''=-(1+(r')^2)\dfrac{(n-k)}{(k+1)r}\left(\dfrac{(k+1)r'r-(n-k-1)}{(n-k)-krr'}\right).
\end{align*}
Therefore, the principal curvatures of the small portion of $W_R$ are given by
\begin{align*}
   \lambda_1=\dfrac{g(r^{-1},r')}{\sqrt{1+(r')^2}}=\dfrac{(n-k)}{(k+1)}\dfrac{\dfrac{(k+1)r'r-(n-k-1)}{(n-k)-krr'}}{r\sqrt{1+(r')^2}}\mbox{ and }\lambda_i=\dfrac{1}{r\sqrt{1+(r')^2}}. 
\end{align*}
Then, we will have that $\lambda\in\Gamma_{k+1}$ if, and only if,
\begin{align*}
   0<\dfrac{(n-k)}{k+1}\dfrac{(k+1)rr'-(n-k-1)}{(n-k)-krr'}+\dfrac{n-l}{l},\: l=1,\ldots, k+1.
\end{align*}
Consequently, since  
\begin{align*}
0<\dfrac{n-l}{l}-\dfrac{n-k-1}{k+1}\Leftrightarrow 0<k+1-l,    
\end{align*}
we obtain 
that $\lambda(0)\in\partial\Gamma_{k+1}\cap\Gamma_k$, or equivalently $$\lambda(0)\in\partial\Gamma_{k+1}\cap\set{\dfrac{\partial Q_{k+1,k}}{\partial\lambda_i}>0},$$ and $\lambda(p)\in\Gamma_k$ for $p\in W_R\setminus\set{0}$. 
\newline

Next, the ODE for $f=Q_{k+1,k}$ is given by
\begin{align}\label{Q_k+1,k-trans R>0}
\begin{cases}
 v'=\dfrac{n-k}{k+1}(1+v^2)\dfrac{v}{r}\dfrac{(k+1)-(n-k-1)\dfrac{v}{r}}{(n-k)\dfrac{v}{r}-k},\: r\geq r_0,
 \\
 v(r_0)=v_0.
\end{cases},\end{align}
Firstly, since $v \to -\infty$ as $r\to R^+$, we may take $r_0$ to be sufficiently close to $R$ so that $v_0<0$. Note that the right hand side of Eq. \eqref{Q_k+1,k-trans R>0} is positive when $v<0$. Thus the solution is increasing as long as this is the case. On the other hand, $v_+=0$ is a solution to the Eq. \eqref{Q_k+1,k-trans R>0}, therefore $v$ and $v_+$ cannot coincide which means that $v$ remains negative for all $r\geq r_0$.
\newline

Now we show that $v \to 0$ as $r \to \infty$. Firstly, since $v$ is increasing and bounded above by $0$, $\displaystyle L \coloneqq\lim_{r\to \infty}v$ exists in $[v_0,0]$. If $L \neq 0$, then there exists $\varepsilon>0$ such that 
\begin{align*}
   \dfrac{v}{r}<\dfrac{-\varepsilon}{r}.
\end{align*}
Then, since $k<n-1$ and the function $g(y)=\dfrac{n-k}{k+1}y\dfrac{(k+1)-(n-k-1)y}{(n-k)y-k}$ is decreasing, we have
\begin{align*}
 v'\geq (1+v^2)g\left(\dfrac{-\varepsilon}{r}\right)\geq (1+v^2)\dfrac{(n-k)\varepsilon}{kr}.
\end{align*}
Therefore, $\tan\left(\dfrac{(n-k)\varepsilon}{k}\ln(r)\right)=O(v)$, but this contradicts that $v$ is bounded. 
\newline

\begin{remark}
   We remark that the asymptotic behavior of the lower branch of the $Q_{k+1,k}$-translator cannot decay to $0$ faster than $\dfrac{-1}{r^{\frac{n-k}{k}(1+\varepsilon)}}$ for every $\varepsilon>0$ and $1\leq k$. 
   \\
   To see this, let $C,\varepsilon>0$ and consider $w_{C,\varepsilon}(r)=-Cr^{-(1+\varepsilon)a}$ where $a=\dfrac{n-k}{k}$.
   Then, $w_{C,\varepsilon}$ is a super-solution to Eq. \eqref{Q_k+1,k-trans R>0} with $v(r_0)<0$. In fact,
\begin{align*}
    w'(r)=\dfrac{(1+\varepsilon)aC}{r^{(1+\varepsilon)a+1}}
\end{align*}
and the RHS of Eq. \eqref{Q_k+1,k-trans R>0} is 
\begin{align*}
    &aC(1+w^2)\dfrac{1}{r^{(1+\varepsilon)a+1}}\dfrac{(k+1)r^{(1+\varepsilon)a+1}+(n-k-1)}{(n-k)+kr^{(1+\varepsilon)a+1}}
    \\
    &\leq aC(1+w^2)\dfrac{1}{r^{(1+\varepsilon)a+1}}
    \\
    &\leq w',
\end{align*}
for $r>r_0$, where $r_0$ is sufficiently large. Then, for each $\varepsilon>0$, we may fix $r_1>r_0$ and then choose C small enough so that $v(r_1)<w_{C,\varepsilon}(r_1)$. Consequently, we have that $w_{C,\varepsilon}>v$ for all $r\geq r_1$.
\end{remark}

\begin{remark}
 By the above remark, we have that the asymptotic behavior of the principal curvatures of the lower branch at infinity behave as $O(-r^{-a-1})$, where $a=\dfrac{n-k}{k}$. Then, it follows that $\lambda\in\Gamma_{k+1}$ if, and only if, 
 \begin{align*}
0<\binom{n-1}{l}\dfrac{(-1)^{l-1}}{r^{l(a+1)}}\dfrac{n(k-l)}{kl}\mbox{, holds for }l=1,\ldots,k+1. 
 \end{align*}
 Therefore, $\lambda\not\in\Gamma_{k+1}$ for any $k$ as $r\to\infty$ of the lower branch. 
\end{remark}

\begin{step}
Finally, the construction of the \emph{wing-like} solution.
\end{step}
 This will be done by extending the small piece $r:(-\varepsilon,\varepsilon)\to\R$ in the two branches by plugging the initial conditions 
\begin{align*}
\begin{cases}
u'\left(r(-\varepsilon)\right)=\dfrac{1}{r'(-\varepsilon)}   
\\
u(r(-\varepsilon))=-\varepsilon.
\end{cases}
\begin{cases}u'\left(r(\varepsilon)\right)=\dfrac{1}{r'(\varepsilon)}   
\\
u(r(\varepsilon))=\varepsilon.
\end{cases}
\end{align*}
in Equations \eqref{S_k-trans R>0} and \eqref{Q_k+1,k-trans R>0}, respectively.

\section{Application: Growth estimate}\label{sec:Application}
In this section we will prove the theorem \ref{T4}, which we repeat here for the convenience of the reader.
\begin{theorem}
    Let $\Sigma=\set{(x,u(x)):x\in\mathbb{R}^n}$ be an entire convex $f$-translator for some nondegenerate speed $f$. Assume further, that there exist  $a,b,C_1,C_2,R>0$ such that 
    \begin{align*}
    C_1|x|^a\leq u(c)\leq C_2|x|^b,\mbox{ for }|x|\geq R,
    \end{align*}
    then, $a\leq 2\leq b$. 

    In addition, if $a=b=2$, then $u(x)$ agrees with the ``bowl''-type solution up to vertical translations. 
\end{theorem}
\begin{proof}
Let assume first that $a>2$, and let $P$ be the ``bowl''-type $f$-translator in $\mathbb{R}^{n+1}$. Recall that $P$ is an entire strictly convex rotationally symmetric graph smoothly asymptotic to
\begin{align*}
 \dfrac{|x|^2}{f(0,1)}-\restr{\dfrac{\partial f}{\partial\lambda_1}}{\lambda=(0,1)}\ln(|x|)+O(|x|^{-1}). 
\end{align*}

Next, by translating suitably $P$ over $\Sigma$, we can find a $t_0>0$ such that $P-te_{n+1}$ lies strictly below from $\Sigma$ for $t\geq t_0$. Note that this is possible since $a>2$. Now, we may translate $P-te_{n+1}$ upward  until touches $\Sigma$ for the first time. Finally, by the interior tangential principle Theorem \ref{T2}, we obtain $\Sigma=P$, but this contradicts $a>2$.

The case $b<2$ is analogous, the only change is to place the ``bowl''-type soliton above $\Sigma$ and move it down until it touches $\Sigma$.

Finally, when $a=b=2$, then $\Sigma$ is smoothly asymptotic to $P$, and by Theorem \ref{T3}, $\Sigma=P$ up to a vertical translation.   
\end{proof}

%\begin{remark}
%It is important to note that if a \emph{nondegenerate} speed function is dominated by the mean curvature, then the unique entire  strictly convex $f$-translators is the ``bowl''-type solution.
%
%Indeed, let $u:\R^n\to\R$ be an entire strictly convex solution to Equation \eqref{f-trans}, and assume that there exit two positive constant $C_1<C_2$ such that
%\begin{align*}
%    C_1H\leq f(\lambda)\leq C_2 H, \mbox{ for } \lambda\in\Gamma. 
%\end{align*}
%Then, it follows that dilations of the ``bowl'' soliton of the mean curvature flow are sub and super-solutions of Equation \eqref{f-trans}, implying that
%\begin{align*}
%    \dfrac{1}{C_2}\left(\dfrac{|x|^2}{n-1}-\ln|x|+O(|x|^{-1})\right)\leq u(x)\leq \dfrac{1}{C_1}\left(\dfrac{|x|^2}{n-1}-\ln|x|+O(|x|^{-1})\right),\mbox{ as }|x|\to\infty, 
%\end{align*}
%but by Theorem \ref{T3} the function $u(x)$ is a vertical translation of the ``bowl''-type solution.
%
%An example of a convex nondegenerate speed functions dominated by the mean curvature is $f(\lambda)=|A|=\sqrt{\lambda_1^2+\ldots+\lambda_n^2}$, since it holds
%\begin{align*}
%\dfrac{H}{\sqrt{n}}\leq|A|\leq H,\mbox{ when }\lambda\in\Gamma_2. 
%\end{align*}
%\end{remark}
\begin{remark}
    Theorem \ref{T4} also holds when the hypotheses of $f$ are changed by the one given in Remark \ref{Tangency Principle}. This means also holds for entire $f$-translators such that the principal curvatures of the graph belongs to $\Gamma$, and $\Gamma$ is a convex cone of $\set{f>0}$ that contains the point $(1,\e)$. 
\end{remark}

\printbibliography

\end{document}